\documentclass[10pt]{amsart}
\usepackage[margin=4cm]{geometry}

\usepackage{color} 

 \usepackage{appendix}

\makeatletter
\let\@wraptoccontribs\wraptoccontribs
\makeatother

\newtheorem{theorem}{Theorem}[section]
\newtheorem{prop}[theorem]{Proposition}
\newtheorem{proposition}[theorem]{Proposition}
\newtheorem{lemma}[theorem]{Lemma}

\newtheorem{corollary}[theorem]{Corollary}
\newtheorem{problem}{Problem}
\newtheorem{conjecture}[problem]{Conjecture}

\theoremstyle{definition}

\theoremstyle{remark}
\newtheorem{remark}[theorem]{Remark}

\newtheorem*{claim*}{Claim}

\newcommand{\bd}{{\partial}}
\newcommand{\bH}{\mathbb{H}}
\newcommand{\reals}{{\mathbb R}}
\newcommand{\natls}{{\mathbb N}}

\newcommand{\zed}{{\mathbb Z}}

\newcommand{\cL}{{\mathcal{L}}}

\newcommand{\cT}{{\mathcal{T}}}
\newcommand{\cG}{{\mathcal{G}}}

\newcommand{\cN}{{\mathcal{N}}}

\newcommand{\cA}{{\mathcal{A}}}
\newcommand{\cB}{{\mathcal{B}}}
\newcommand{\cE}{{\mathcal{E}}}

\newcommand{\cR}{{\mathcal{R}}}
\newcommand{\cD}{{\mathcal{D}}}
\newcommand{\cH}{{\mathcal{H}}}

\newcommand{\C}{{\mathbb{C}}}
\newcommand{\R}{{\mathbb{R}}}
\newcommand{\Z}{{\mathbb{Z}}}

\newcommand{\cM}{{\mathcal{M}}}
\newcommand{\cP}{{\mathcal{P}}}

\newcommand{\N}{\mathbb{N}}

\newcommand{\nn}[1]{(\ref{#1})}
\newcommand{\om}{\omega}

\renewcommand{\part}[2]{\frac{\partial #2}{\partial x_{#1}}}

\newcommand{\op}{\operatorname}
\newcommand{\acou}[2]{\langle{#1},{#2}\rangle}

\numberwithin{equation}{section}

\begin{document}

\title[Conformal invariants from nodal sets. I]{Conformal invariants
from nodal sets. I. Negative Eigenvalues and Curvature Prescription}
 
\contrib[With an Appendix by]{A.\ Rod Gover and Andrea {Malchiodi*}}

\author[Y. Canzani]{Yaiza Canzani}
\address{Department of Mathematics and Statistics, McGill University, Montr\'eal,
Ca\-na\-da.} 
\email{canzani@math.mcgill.ca}

\author[A.R. Gover]{A.\ Rod Gover}
\address{Department of Mathematics, University of Auckland, New Zealand \&
Mathematical Sciences Institute, Australian National University, Canberra, 
Australia.}
\email{gover@math.auckland.ac.nz}

\author[D. Jakobson]{Dmitry Jakobson}
\address{Department of Mathematics and
Statistics, McGill University, Montr\'eal, Ca\-na\-da.}
\email{jakobson@math.mcgill.ca}

\author[R. Ponge]{Rapha\"el Ponge}
\address{Department of Mathematical Sciences, Seoul National University, 
Seoul, Korea.} 
\email{ponge.snu@gmail.com}



\thanks{Y.C.\ was supported by Schulich Fellowship of McGill University (Canada). 
A.R.G.\ was supported by Marsden Grant 10-UOA-113 (New Zealand). D.J.\ was
  supported by NSERC and FQRNT grants and Dawson Fellowship of McGill 
University (Canada).  
  R.P.\ was supported by JSPS Grant-in-Aid (Japan) and Research Resettlement 
Fund of Seoul National University (Korea).\\
\indent   *AM: SISSA, Trieste, Italy. \emph{E-mail address}: malchiod@sissa.it}

\begin{abstract}
In this paper, we study conformal invariants that arise from nodal sets and negative eigenvalues of conformally 
covariant operators; more specifically, the GJMS operators, which include the Yamabe and Paneitz operators. We give 
several  applications to curvature prescription problems. We establish a version in conformal geometry of Courant's Nodal Domain Theorem. 
We also show that on any manifold of dimension
$n\geq 3$, there exist many metrics for which our invariants are nontrivial. 
We prove that the Yamabe operator can have an arbitrarily large number of 
negative eigenvalues on any manifold of dimension $n\geq 3$. We obtain
similar results for some higher order GJMS operators on some Einstein and Heisenberg manifolds. We describe the invariants 
arising from the Yamabe and Paneitz operators associated to left-invariant metrics on Heisenberg manifolds. Finally, in 
the appendix, the 2nd named author and Andrea Malchiodi study the $Q$-curvature prescription problems for
non-critical $Q$-curvatures.
\end{abstract}

\maketitle

\section{Introduction} 

Nodal sets (i.e., zero loci) of eigenfunctions were first considered in the 18th century by Ernst 
Chladni in his 1787 paper \emph{Entdeckungen \"uber die Theorie des Klanges} on 
vibrating plates. 
 More recently, some important results about nodal sets were obtained by 
Courant~\cite{Courant},  
 Pleijel~\cite{Pleijel}, Cheng-Yau~\cite{Cheng-Yau} and 
Donnelly-Fefferman~\cite{DF}, among others.  For high energy eigenfunctions of the Laplacian geometry and topology of nodal 
sets and nodal domains (i.e., connected components of complements of nodal sets) have 
also been studied in quantum chaos, in particular in connection to random wave theory (see, e.g., \cite{BS, NS, TZ}).

Conformally invariant operators with leading term a power of the
Laplacian $\Delta_{g}$ have been central in mathematics and physics
for over 100 years. The earliest known of these is the conformally
invariant wave operator which was first constructed for the study of
massless fields on curved spacetime (see, e.g., Dirac~\cite{Dirac}). Its
Riemannian signature elliptic variant, usually called the Yamabe
operator, controls the transformation of the Ricci scalar curvature
under conformal rescaling and so plays a critical role in the Yamabe
problem on compact Riemannian manifolds. A conformal operator with
principal part $\Delta_{g}^{2}$ is due to Paneitz~\cite{Paneitz}, and
sixth-order analogues were constructed by Branson~\cite{Br} and W\"unsch~\cite{Wu}.

Two decades ago Graham, Jenne, Mason and Sparling (GJMS) solved a
major existence problem in~\cite{GJMS}, where they used the ambient
metric of Fefferman-Graham~\cite{FG85, FG12} to show the existence of
conformally invariant differential operators $P_{k,g}$ (to be referred
to as the GJMS operators) with principal part $\Delta_{g}^{k}$.  In
odd dimensions, $k$ is any positive integer, while in dimension $n$
even, $k$ is a positive integer no more than $\frac{n}{2}$. The $k =
1$ and $k = 2$ cases recover the Yamabe and Paneitz operators,
respectively. Furthermore, the GJMS operators are intimately related to the
$Q_{k}$-curvatures $Q_{k,g}$ identified by Branson~\cite{BrSeoul,Tomsharp} 
(see also Section~\ref{section:Qcurv}); the $Q_{\frac{n}{2},g}$-curvature is also known as 
Branson's $Q$-curvature. 

The aim of this article is to study conformal invariants arising from
nodal sets and negative eigenvalues of GJMS operators. In particular,
we give some applications to curvature prescription problems.  Various
authors have considered spectral theoretic functions associated to
conformally covariant operators, e.g., Parker-Rosenberg~\cite{PR},
Osgood-Phillips-Sarnak~\cite{OPS}, Branson-{\O}rsted~\cite{BrO91a},
Branson-Chang-Yang~\cite{BCY}, Chang-Yang~\cite{CY}, and
Okikiolu~\cite{Ok}.  However, to our knowledge this is the first time
that nodal sets have been considered generally in the setting of
conformal geometry.

A first observation is that nodal sets and nodal domains of any
null-eigenfunction of a GJMS operator are conformal invariants
(Proposition~\ref{prop:nodal-sets-Pk}).  In case of the critical GJMS
operator $P_{\frac{n}{2},g}$, it can be further shown this feature is
actually true for any level set
(Proposition~\ref{prop:nodal-sets-Pcritical}). Notice that these
results actually hold for more general conformally invariant
operators, including the fractional conformal powers of the Laplacian
(see Remark~\ref{rem:Nodal.general}).

Here we also look at the negative eigenvalues of the GJMS operators. In
particular, we show that the number of negative eigenvalues of a GJMS
operator is a conformal invariant
(Theorem~\ref{prop:confinv-negative-eigenvalues}). It was shown by
Kazdan-Warner~\cite{KW75} that the sign of the first eigenvalue of the
Yamabe operator is a conformal invariant. We prove that this result
actually holds for all GJMS operators
(Theorem~\ref{thm:negative.sign-first-eigenvalue}). Once again these
results hold for general conformally invariant operators
(see~Remark~\ref{rem:negative.general}) acting between the same spaces. 

A natural question is whether, for a given operator, the number of negative
eigenvalues can become arbitrarily large as the conformal class
varies. We prove that this indeed the case for the Yamabe operator on
any connected manifold (Theorem~\ref{neg:eig:large}). The proof relies
on a deep existence result of Lokhamp~\cite{Lo96}. We give a more
explicit proof in case of products with hyperbolic surfaces (see
Proposition~\ref{prop:negative.Yamabe.product-hyperbolic-surface}).
Furthermore, on the product of a hyperbolic manifold with a hyperbolic
surface we construct hyperbolic metrics for which various higher order
GJMS operators have arbitrary large numbers of negative eigenvalues
(see Theorem~\ref{CriticalGJMS:hyperbolic} for the precise statement).
In addition, we prove a version of Courant's nodal domain theorem in
conformal geometry: if the Yamabe operator has  $m$ negative
eigenvalues, then its null-eigenfunctions have at most $m+1$~nodal
domains (Theorem~\ref{number:nodal}).

The problem of prescribing the curvature (Gaussian or scalar) of a
given compact manifold is very classical and is known as the {\em
  Kazdan-Warner problem} (see \cite{Au, B, KW} and the references
therein).  The extension of this question to Branson's $Q$-curvature
has proved to be an important proxblem for the development of
mathematical ideas (see, e.g.,
\cite{BaFaRe,Brendle,CGY,CY,DM,DR,MalStr,Nd:CQMAD}).

We look at some constraints on curvature prescription in terms of nodal 
sets. A main result is Theorem~\ref{yamabe on nodal domains} which states that, 
if $u$ is a null-eigenfunction for the Yamabe operator $P_{1,g}$ and 
$\Omega$ is a nodal domain of $u$, then
\begin{equation*}
    \int_{\Omega} |u|\,P_{1,g}(v) \, dv_g=- \int_{\partial \Omega}
v\, \| {}^{g}\nabla u \|_g\, d\sigma_g \qquad \forall v \in C^{\infty}(M,\R).
\end{equation*}
Another main result is Theorem~\ref{int:negative} which
asserts that, if a function $f$ is the scalar curvature of some metric
in the conformal class of $g$, then there is a smooth function
$\omega>0$ such that, for any null-eigenfunction $u$ of the Yamabe
operator and any nodal domain $\Omega$ of $u$,
\begin{equation*}
    \int_{\Omega}f|u| \omega \, dv_{g}<0. 
\end{equation*}
As a corollary, we see that, if $R_{\hat{g}}$ is the scalar curvature
of some metric in the conformal class of $g$, then $R_{\hat{g}}$
cannot be positive everywhere on $\Omega$.

We illustrate our results on nodal sets and negative eigenvalues in
the case of the Yamabe and Paneitz operators associated to left-invariant
metrics on a  Heisenberg manifold $\Gamma\backslash \bH_{d}$, obtained
as the quotient of the $(2d+1)$-dimensional Heisenberg $\bH_{d}$ group
by some lattice $\Gamma$. Using the representation theory of the
Heisenberg group, we are able to give spectral resolutions for the
Yamabe and Paneitz operators (see Proposition~\ref{Yamabe:nil} and
Proposition~\ref{prop:spectrum-Paneitz-Heisenberg}). Interestingly
enough, some of the eigenfunctions involve theta-functions. As a
result, this enables us to explicitly describe their nodal sets (see
Proposition~\ref{prop:negative-eigenvalues-Yamabe-Heisenberg}).
Furthermore, we can give lower bounds for the number of negative
eigenvalues of the Yamabe and Paneitz operators, which shows that
these operators can have an arbitrarily large number of negative
eigenvalues; the bounds involve the volume of $\Gamma\backslash
\bH_{d}$ (see Proposition~\ref{prop:neg-Yamabe-Heisenberg} and
Proposition~\ref{Paneitz:neg:heisenberg}).

For Branson's $Q$-curvature it was shown by
Malchiodi~\cite{MalchSIGMA} (when $\int Q=0$) and Gover~\cite{Gov10}
(for the general case) that if the kernel of the critical GJMS
operator contains nonconstant functions, then there is an
infinite-dimensional space of functions that cannot be the
$Q$-curvature of any metric in the conformal class.

In the appendix, by the second named author and Andrea Malchiodi, it
is shown that, surprisingly, similar results are available for the
non-critical $Q$-curvatures; these are the curvature quantities
$Q_{k,g}$, $k\neq \frac{n}{2}$.
The main result is Theorem \ref{main} which proves that for $0\neq
u\in \ker P_{k,g}$ any function $s_u$ on $M$, with the same strict
sign as $u$, cannot be $Q_{k,\widehat{g}}$ for any metric
$\widehat{g}$ in the conformal class. In particular this potentially
obstructs achieving constant $Q_{k}$-curvature (see Theorem
\ref{constants}). Theorem \ref{main} is also used to identify a space
$\mathcal{I}$ of functions, determined by the conformal structure (and
in general properly contained in $C^\infty(M,\mathbb{R})$), which
contains the range of $Q_{k,g}$, as $g$ ranges over the conformal
class (see Theorem \ref{consthm} for the precise statement).

Various open problems and conjectures are gathered in Section~\ref{section: open problems}. Further invariants will be considered 
in a forthcoming paper \cite{CGJP}. 

The paper is organized as follows.  In Section \ref{section:Qcurv}, we
review the main definitions and properties of the GJMS operators and
$Q$-curvatures.  In Section \ref{section: nodal sets}, we study the
nodal sets of GJMS operators.  In Section
\ref{section:negative-eigenvalues}, we study the negative eigenvalues
of GJMS operators. In Section \ref{section:scalar:sign}, we discuss
curvature prescription problems.  In Section \ref{section:
  heisenberg}, we study the nodal sets and negative eigenvalues of
Yamabe and Paneitz associated to left-invariant metrics on Heisenberg
manifolds. In Section~\ref{section: open problems}, we present various
open problems and conjectures.  Finally, the appendix by the second named
author and Andrea Malchiodi deals with $Q$-curvature prescriptions for
non-critical $Q$-curvatures.

The results of this paper were announced in~\cite{ERA}. 

\section{GJMS operators and $Q$-curvatures}\label{section:Qcurv}
Let $M$ be a Riemannian manifold of dimension $n\geq 3$. 
A \emph{conformally covariant differential operator} of biweight $(w,w')$ is a covariant 
differential operator
$P_{g}$ such that, under any conformal change of metric $\hat{g}:=e^{2\Upsilon}g$, 
$\Upsilon\in C^{\infty}(M, \R)$, it transforms according to the formula
   \begin{equation}\label{conf:covariant:change}
       P_{\hat{g}}=e^{-w'\Upsilon}P_{g}e^{w\Upsilon}.
   \end{equation}

An important example of a conformally invariant differential operator is the \emph{Yamabe operator} (a.k.a.~conformal
Laplacian),
  \begin{equation}
        P_{1,g}:=\Delta_{g}+\frac{n-2}{4(n-1)}R_g,
        \label{eq:Yamabe-operator}
  \end{equation}
where $R_g$ is the scalar curvature. This is a conformally invariant
operator of biweight $\left(\frac{n}{2}-1,\frac{n}{2}+1\right)$.

Another example is the \emph{Paneitz operator},
  \begin{equation}
    P_{2,g}:=\Delta_{g}^{2}+\delta  V d+\frac{n-4}{2}\left\{
    \frac{1}{2(n-1)}\Delta_{g}R_g+\frac{n}{8(n-1)^{2}}R_g^{2}-2|S|^{2}\right\},
    \label{eq:Paneitz-operator}
\end{equation}where $S_{ij}=\frac{1}{n-2}(\op{Ric_g}_{ij}-\frac{R_g}{2(n-1)}g_{ij})$
 is the Schouten-Weyl tensor and $V$ is the tensor
$V_{ij}=\frac{n-2}{2(n-1)} R_g g_{ij}-4S_{ij}$ acting on 1-forms
(i.e., $V(\omega_{i}dx^{i})=(V_{i}^{~j}\omega_{j})dx^{i}$). The Paneitz operator is a 
conformally invariant operator of biweight
$\left(\frac{n}{2}-2,\frac{n}{2}+2\right)$.

A generalization of the Yamabe and Paneitz operators is provided by
the GJMS operators.  They were constructed by Graham, Jenne, Mason and
Sparling in~\cite{GJMS} by using the ambient metric of
Fefferman-Graham~\cite{FG85, FG12} (see also~\cite{GP,GZ,Ju} for
 formulas for, and features of, the GJMS operators).

\begin{proposition}[\cite{GJMS}]
   For $k=1,\ldots,\frac{n}{2}$ when $n$ is even, and for all
   non-negative integers $k$ when $n$ is odd, there is a conformally
   invariant operator $P_{k}=P_{k,g}$ of biweight
   $\left(\frac{n}{2}-k,\frac{n}{2}+k\right)$ such that
\begin{equation}
        P_{k,g}=\Delta_{g}^{(k)} + \ \text{lower order terms}.
\end{equation}
\end{proposition}

When $n$ is even the ambient metric is obstructed at finite order by
Fefferman-Graham's obstruction tensor (see~\cite{FG85, FG12}). This is
a conformally invariant tensor which in dimension 4 agrees with the
Bach tensor.  As a result, the ambient metric construction of the GJMS
operators $P_{k}$ breaks down for $k>\frac{n}{2}$, when $n$ is
even. In fact, as proved by Graham~\cite{Gnon} in dimension 4 for
$k=3$ and by Gover-Hirachi~\cite{GH} in general, there do not exist
conformally invariant operators with same leading part as
$\Delta^{k}_{g}$ for $k>\frac{n}{2}$ when $n$ is even. For this
reason, the operator $P_{\frac{n}{2},g}$ is sometimes called the
\emph{critical GJMS operator}. Notice that for $P_{\frac{n}{2}}$ the
transformation law becomes
\begin{equation*}
    P_{\frac{n}{2},e^{2\Upsilon}g}=e^{-n\Upsilon}P_{\frac{n}{2},g} \qquad 
\forall \Upsilon\in C^{\infty}(M,\R).
\end{equation*}

When $n$ is even and the metric $g$ is conformally Einstein, the
Fefferman-Graham obstruction tensor vanishes, and a canonical ambient
metric exists all orders (in fact it exists on a collar), see
\cite{FG12} and references therein. It is also the case that on an
even dimensional conformally Einstein manifold the GJMS operator family may be
extended to all (even) orders in a canonical way \cite{Gov06}.  Furthermore,
when $g$ is actually Einstein, say $\op{Ric}_{g}=\lambda(n-1)g$ for
some $\lambda\in \R$, it was shown by Graham \cite{Gr03, FG12} and
Gover~\cite{Gov06} (see also Guillarmou-Naud~\cite{GN} for constant
sectional curvature spaces) that
\begin{equation}
    P_{k,g}=\prod_{1\leq j\leq k}\left(\Delta_{g}+\frac{\lambda}{4}(n+2j-2)(n-2j) \right).
    \label{eq:GJMS-Einstein}
\end{equation}

The GJMS operators $P_{k,g}$ are formally self-adjoint
(see~\cite{GZ,FG02}). Moreover, they are intimately related to the
$Q_{k}$-curvatures identified by Branson~\cite{BrSeoul,Tomsharp}.
For $k=1,\cdots , \frac{n}{2}-1$ when $n$ is even and
for $k \in \N_{0}$ when $n$ is odd, the $Q_{k}$-curvature is defined
by
\begin{equation}\label{Qk:def}
    Q_{k}=Q_{k,g}:=\frac{2}{n-2k}P_{k,g}(1).
\end{equation}
When $n$ is even, the $Q_{\frac{n}{2}}$-curvature  (a.k.a.~Branson's  
$Q$-curvature) is defined by analytic
continuation arguments (see~\cite{BrO,Br}; see also \cite{GZ,FG02}).

For instance, it follows
from~(\ref{eq:Yamabe-operator})--(\ref{eq:Paneitz-operator}) that
\begin{equation*}
    Q_{1,g}=\frac{1}{2(n-1)}R_{g} \qquad \text{and} \qquad Q_{2}=
    \frac{1}{2(n-1)}\Delta_{g}R_{g}+\frac{n}{8(n-1)^{2}}R_{g}^{2}-2|S_{2,g}|^{2}.
\end{equation*}

As explained in \cite{BG},
\begin{equation}\label{Qk:curv}
    P_{k,g}=\delta S_{k,g} d+\frac{n-2k}{2}Q_{k,g},
\end{equation}
where $S_{k,g}$ is an operator acting on $1$-forms. In particular, we
see that the critical GJMS operator $P_{\frac{n}{2}}$ kills the
constant functions. It follows from (\ref{Qk:curv}) that, when $k\neq
\frac{n}{2}$, under a conformal change of metric
$\hat{g}:=e^{2\Upsilon}g$, $\Upsilon\in C^{\infty}(M, \R)$, 
\begin{equation*}
    Q_{k,\hat{g}}=e^{ -2k\Upsilon}Q_{k,g}+\frac{2}{n-2k}e^{-\Upsilon\left(\frac{n}{2}+k\right)}\delta S_{k,g}d 
e^{\Upsilon \left(\frac{n}{2}-k\right)}.
\end{equation*}
When $n$ is even, for $k=\frac{n}{2}$, we have
\begin{equation}\label{transf:Qcurv}
Q_{\frac{n}{2},\hat{g}}=e^{-n\Upsilon}Q_{\frac{n}{2},g}+e^{-n\Upsilon}P_{\frac{n}{2},g}(\Upsilon).
\end{equation}

Finally, let us mention that there is a rather general theory for the
existence of linear conformally invariant differential operators due
to Eastwood-Slov\'ak~\cite{ES}. Further conformally invariant
differential objects were also constructed by, e.g., Alexakis~\cite{Al03,Al06} and Juhl~\cite{Ju}.


\section{Nodal sets of GJMS operators}\label{section: nodal sets}
In this section, we shall look at the conformal invariance of nodal sets (i.e., zero-loci) and nodal domains (i.e., connected components of complements of nodal sets) of eigenfunctions of GJMS operators. 

Throughout this section we let $(M^{n},g)$ be a Riemannian manifold ($n\geq 3$). In addition, we let $k\in \N_{0}$ and further 
assume $k\leq \frac{n}{2}$ when $n$ is even. 

It is convenient to look at conformally covariant scalar operators as 
linear operators between spaces of conformal densities. Throughout the
sequel we shall regard a conformal density of weight $w$, $w \in \R$,
as a family $(u_{\hat{g}})_{\hat{g}\in [g]}\subset C^{\infty}(M)$
parametrized by the conformal class $[g]$ in such way that
\begin{equation*}
    u_{e^{2\Upsilon}g}(x)=e^{-w\Upsilon(x)}u_{g}(x) \qquad \forall \Upsilon \in C^{\infty}(M,\R). 
\end{equation*}
We shall denote by $\cE[w]$ the space of conformal densities of weight $w$. 

The space $\cE[w]$ can be realized as the space of smooth functions of a line bundle over $M$ 
as follows (see 
also~\cite{PR}). Denote by 
$\op{CO}(n)$ the conformal group of $\R^{n}$, that is, the subgroup of $\op{GL}_{n}(\R)$ 
consisting of positive scalar 
multiples of orthogonal matrices. The datum of the conformal class $[g]$ gives rise to a 
reduction of the structure 
group of $M$ to the conformal group $\op{CO}(n)$. Denote by $E[w]$ the line bundle over $M$ 
associated to the 
representation $\rho_{w}:\op{CO}(n)\rightarrow \R^{+}_{*}$ given by 
\begin{equation*}
    \rho_{w}(A)=|\det A|^{\frac{w}{n}} \qquad \forall A \in \op{CO}(n).
\end{equation*}
Any metric $\hat{g}=e^{2\Upsilon}g$, $\Upsilon\in C^{\infty}(M)$, in the conformal class $[g]$ defines a 
global trivialization $\tau_{\hat{g}}:E[w]\rightarrow M\times \R$ with transition map, 
\begin{equation*}
 \tau_{\hat{g}}\circ  \tau_{g}^{-1}(x)=e^{w\Upsilon(x)} \qquad \forall x \in M.   
\end{equation*}
This gives rise to a one-to-one correspondence between smooth sections of  $E[w]$ and conformal 
densities. Namely, 
to any $u \in C^{\infty}(M,E[w])$ corresponds a unique conformal density $(u_{\hat{g}})_{\hat{g}\in [g]}$ 
in $\cE[w]$ 
such that, for any metric $\hat{g} \in [g]$,
\begin{equation*}
    \tau_{\hat{g}}\circ u(x)=(x,u_{\hat{g}}(x)) \qquad \forall x \in M.
\end{equation*}

The property that the GJMS operator operator $P_{k,g}$ is conformally
invariant of biweight $\left(\frac{n}{2}-k,\frac{n}{2}+k\right)$
exactly means it gives rise to a linear operator,
\begin{equation*}
    P_{k}:\cE\left[ -\frac{n}{2}+k\right] \rightarrow \cE\left[ -\frac{n}{2}-k\right] , 
\end{equation*}such that, for all $u=(u_{\hat{g}})_{\hat{g}\in [g]}$ in $\cE\left[ -\frac{n}{2}+k\right]$, 
\begin{equation*}
 (P_{k}u)_{\hat{g}} (x)=  (P_{k,\hat{g}}u_{\hat{g}})(x) \qquad \forall \hat{g}\in [g] \ \forall x \in M.
\end{equation*}
In particular, this enables us to regard the
nullspace of $P_{k,g}$ as a space of conformal
densities. Clearly the dimension of $\ker P_{k,g}$ is an
invariant of the conformal class $[g]$.

We observe that if $u=(u_{\hat{g}})_{\hat{g}\in [g]}$ is a conformal density of weight $w$, then the nodal set 
the zero locus $u^{-1}_{\hat{g}}(0)$ is independent of the metric $\hat{g}$, and
hence is an invariant of the conformal class $[g]$. Applying this observation to null-eigenvectors of $P_{k}$ we then get

\begin{proposition}\label{prop:nodal-sets-Pk}~ Let $k\in \N$ and further assume 
$k\leq \frac{n}{2}$ if $n$ is even. 
    \begin{enumerate}
        \item If $\dim \ker P_{k,g}\geq 1$, then the nodal sets and
          nodal domains of any nonzero null-eigenvector of $P_{k,g}$ give rise to invariants 
of the conformal class
          $[g]$.    
        \item If $\dim \ker P_{k,g}\geq 2$, then (non-empty)
          intersections of nodal sets of null-eigenvectors of
          $P_{k,g}$ and their complements are invariants of the
          conformal class~$[g]$.
    \end{enumerate}
\end{proposition}


\begin{remark}
A connected component $X$ of an intersection of $p$ nodal sets 
should generically be a co-dimension $p$ submanifold 
of $M$, and in the case it is, the corresponding homology class in $H_{n-p}(M)$ 
would be a conformal invariant.  Further interesting conformal invariants should arise from considering
the topology of $M\setminus X$.  For example, if $\dim M=3$ and
$\dim\ker P_k=2$, and $u_1,u_2\in\ker P_k$, then
$\cN(u_1)\cap\cN(u_2)$ should define a ``generalized link'' in $M$,
and all topological invariants of that set and its complement in $M$
would be conformal invariants.  Related invariants are considered in
\cite{Ch3,CR}.
\end{remark}

When $k=\frac{n}{2}$ ($n$ even) the nullspace of $P_{\frac{n}{2}}$ is contained in the space 
$\cE[0]$ of conformal densities of 
weight $0$ and it always contains constant functions (seen as conformal densities of weight 
zero, i.e., a constant 
family of constant functions).

Observe also that if $u=(u_{\hat{g}})_{\hat{g}\in [g]}$ is a conformal density of weight 0, then, 
in addition to the 
zero-locus, all the level sets $\{x\in M;\ u_{g}(x)=\lambda\}$, $\lambda \in \C$, are 
independent of the representative metric $g$. 
Therefore, for the critical GJMS operator we obtain

\begin{proposition}\label{prop:nodal-sets-Pcritical}
Assume $n$ is even. If $\dim \ker P_{\frac{n}{2}}\geq 2$, then the
level sets of any non-constant null-eigenvector of $P_{\frac{n}{2}}$ are invariants of the 
conformal class $[g]$.
\end{proposition}

Next, we mention the following result. 
\begin{proposition}\label{prop:nodal-sets-integral}
Assume $M$ compact and $k<\frac{n}{2}$. Let $u_{g}\in\ker P_{k,g}$ and let us regard 
$u$ as a conformal density of weight $-\frac{n}{2}+k$. Then the integral ${\displaystyle
 \int_{M}|u_{g}(x)|^{\frac{2n}{n-2k}}\, dv_{g}(x)}$ is an invariant of the
conformal class $[g]$.
\end{proposition}
\begin{proof}
Let $\hat{g}=e^{2\Upsilon}g$, $\Upsilon \in C^{\infty}(M,\R)$, be a metric in the conformal class $[g]$. Then 
\begin{equation*}
  \int_{M}|u_{\hat{g}}(x)|^{\frac{2n}{n-2k}}\, dv_{\hat{g}}(x)=  
  \int_{M}\left|e^{\frac{2k-n}{2}\Upsilon(x)}u_{g}(x)\right|^{\frac{2n}{n-2k}}e^{n\Upsilon(x)}\, dv_{g}(x) = 
  \int_{M}|u_{g}(x)|^{\frac{2n}{n-2k}}\, dv_{g}(x). 
\end{equation*}
This proves the result.     
\end{proof}

\begin{remark}\label{rem:Nodal.general}
    Although stated for GJMS operators, the results of this section
    actually hold for any conformally covariant operator that yields
    an endomorphism on some function space.  They even hold for
    conformally invariant pseudodifferential operators, including the
    conformal fractional powers of the Laplacian $P_{s,g}$, $s>0$,
    which extend the GJMS construction to non-integer orders
    (see~\cite{GZ,GQ}). More precisely,
    Proposition~\ref{prop:nodal-sets-Pk} holds \emph{verbatim} for any
    such conformally invariant operator, and
    Proposition~\ref{prop:nodal-sets-Pcritical} (resp.,
    Proposition~\ref{prop:nodal-sets-integral}) holds \emph{verbatim}
    for any such conformally invariant operator of biweight $(w,w')$
    with $w=0$ (resp., $w=\frac{n}{2}-k$ with $k\in
    \left(0,\frac{n}{2}\right)$).
\end{remark}


\section{Negative eigenvalues of GJMS operators}\label{section:negative-eigenvalues}
In this section, we look at the negative eigenvalues of GJMS operators. Throughout this 
section $M^{n}$ is a compact 
manifold ($n\geq 3$) and we let $k \in \N$ (and further assume $k\leq \frac{n}{2}$ when $n$ 
is even). 

Let $\cG$ be the set of Riemannian metrics on $M$ equipped with its standard Fr\'echet-space 
$C^{\infty}$-topology. As mentioned in 
Section~\ref{section:Qcurv}, given any metric $g$ on $M$, the 
GJMS operator $P_{k,g}$ is self-adjoint with respect to the inner product defined by $g$. 
Moreover, as it has same 
leading part as $\Delta_{g}^{k}$ it has a positive principal symbol. Therefore, it spectrum 
consists of a sequence 
of real eigenvalues converging to $\infty$. We thus can order the eigenvalues of $P_{k,g}$ as a 
non-decreasing 
sequence, 
\begin{equation*}
    \lambda_{1}(P_{k,g})\leq \lambda_{2}(P_{k,g}) \leq \cdots,
\end{equation*}where each eigenvalue is repeated according to multiplicity. Notice that by 
the min-max principle,
\begin{equation}
    \lambda_{j}(P_{k,g})=\inf_{\substack{E \subset C^{\infty}(M)\\ \dim E=j}}\sup_{\substack{u \in E\\ \|u\|=1}} 
    \acou{P_{k,g}u}{u}.
    \label{eq:min-max}
\end{equation}

\begin{lemma}[{\cite[Theorem~2]{KS:DCAS3}}]\label{lem:continuity-eigenvalues}
    For every $j\in \N$, the function  $g \rightarrow \lambda_{j}(P_{k,g})$ is continuous on
$\cG$.
\end{lemma}

For a metric $g \in \cG$, we define
\begin{equation*}
    \nu_{k}(g):=\#\{ j\in \N; \ \lambda_{j}(P_{k,g})<0\}.
\end{equation*}
In addition, for any $m \in \N$, we set
\begin{equation*}
    \cG_{k,m}:=\left\{ g \in \cG; \ \text{$P_{k,g}$ has at least $m$ negative eigenvalues}\right\}.
\end{equation*}
Notice that  $\cG_{k,m}=\{g\in \cG;  \nu_{k}(g)\geq m\}=\left\{ g \in \cG; \ 
\lambda_{m}(P_{k,g})<0\right\}$, so it follows from Lemma~\ref{lem:continuity-eigenvalues} that 
$\cG_{k,m}$ is an open subset of $\cG$. 

\begin{theorem}\label{prop:confinv-negative-eigenvalues}
Let $g \in \cG$. Then $\nu_{k}(g)$ is an invariant of the conformal class $[g]$.
\end{theorem}
\begin{proof}
Let $g \in \cG$, and set $m=\nu_{k}(g)$ and $l=\dim \ker
P_{k,g}$. Thus $\lambda_{j}(P_{k,g})<0$ for $j\leq m$, and
$\lambda_{j}(P_{k,g})=0$ for $j=m+1,\cdots, m+l$, and
$\lambda_{j}(P_{k,g})>0$ for $j \geq m+l+1$.  Let $\delta$ be a
positive real number~$<\op{min}\left\{|\lambda_{m}(P_{k,g})|,
\lambda_{m+l+1}(P_{k,g})\right\}$. It follows from
Lemma~\ref{lem:continuity-eigenvalues}, if a metric $\hat{g}$ in the
conformal class $[g]$ is close enough to $g$, then
$\lambda_{m}(P_{k,\hat{g}})<-\delta$ and
$\lambda_{m+l+1}(P_{k,\hat{g}})>\delta$. That is, the only
$\lambda_{j}(P_{k,\hat{g}})$ that are contained in the interval
$[-\delta,\delta]$ are $\lambda_{m+1}(P_{k,\hat{g}}), \cdots,
\lambda_{m+l}(P_{k,\hat{g}})$.

As mentioned in Section~\ref{section: nodal sets}, the dimension of $\ker P_{k,g}$ is an 
invariant of the conformal class $[g]$, so $\dim 
\ker P_{k,\hat{g}}=l$, i.e., there are exactly $l$ of the $\lambda_{j}(P_{k,\hat{g}})$ that are 
equal to $0$. Since we 
know that the $l$ eigenvalues $\lambda_{m+1}(P_{k\hat{g}}), 
\cdots,  \lambda_{m+l}(P_{k\hat{g}})$ are the only eigenvalues that are contained in 
$[-\delta,\delta]$, it follows that $\lambda_{m+1}(P_{k,\hat{g}})=\cdots = 
\lambda_{m+l}(P_{k,\hat{g}})=0$. As $\lambda_{m}(P_{k,\hat{g}})<0$, we then conclude that 
$\nu_{k}(\hat{g})=m$. 

All this shows that the map $g \rightarrow \nu_{k}(g)$ is locally constant when restricted 
to the conformal class 
$[g]$. As $[g]$ is a connected subset of $\cG$ (since this is the range of $C^{\infty}(M,\R)$ 
under $\Upsilon \rightarrow 
e^{2\Upsilon}g$), we deduce that $\nu_{k}(g)$ is actually constant along the conformal class $[g]$. 
This proves the theorem. 
\end{proof}

It follows from Theorem~\ref{prop:confinv-negative-eigenvalues} that the number of negative 
eigenvalues of each GJMS operator defines a partition 
of the set of conformal classes. 

A result of Kazdan-Warner~\cite[Theorem 3.2]{KW75} asserts that the sign of
the first eigenvalue $\lambda_1(P_{1,g})$ is an invariant of the conformal class $[g]$. Notice that 
\begin{enumerate}
    \item[(i)] $\lambda_{1}(P_{k,g})<0$ if and only if $\nu_{k}(g)\geq 1$. 

    \item[(ii)] $\lambda_{1}(P_{k,g})=0$ if and only if $\nu_{k}(g)=0$ and $\dim \ker P_{k,g}\geq 1$.

    \item[(iii)] $\lambda_{1}(P_{k,g})>0$ if and only if $\dim \ker P_{k,g}=\nu_{k}(g)=0$. 
\end{enumerate}
Therefore, as an immediate consequence of the conformal invariance of $\dim \ker P_{k,g}$ and $\nu_{k}(g)$, 
we obtain the following extension of Kazdan-Warner's result.

\begin{theorem}\label{thm:negative.sign-first-eigenvalue}
    The sign of the first eigenvalue $\lambda_{1}(P_{k,g})$ is an invariant of the conformal 
class $[g]$.
\end{theorem}

\begin{remark}
 Let $g_{0}$ be a metric of constant scalar curvature in the conformal
 class. As the nullspace of the Laplacian consists of constant
 functions,
 $\lambda_{1}(P_{1,g_{0}})=\frac{n-2}{4(n-1)}R_{g_{0}}$. Therefore,
 the sign of $\lambda_{1}(P_{1,g})$ agrees with that of the constant
 scalar curvature $R_{g_{0}}$. We also see that
 $\lambda_{1}(P_{k,g})=0$ if and only if $R_{g_{0}}=0$.  Furthermore,
 in that case $\ker P_{1,g_{0}}$ consists of constant functions and
 $\ker P_{1,g}$ is spanned by a single positive function.
\end{remark}
 
\begin{remark}\label{rem:negative.general}
    Both Theorem~\ref{prop:confinv-negative-eigenvalues} and
    Theorem~\ref{thm:negative.sign-first-eigenvalue} hold
    \emph{verbatim} for any conformally invariant pseudodifferential
    operator acting on functions or even on sections of a vector
    bundle. In particular, they hold for the fractional conformal
    powers of the Laplacian on functions.
\end{remark}

The question that naturally arises is whether $\nu_{k}(g)$ can be arbitrary large as $g$ ranges 
over metrics on $M$. The following shows this is indeed the case 
when $k=1$ (i.e., $P_{k,g}$ is the Yamabe operator). 

\begin{theorem}\label{neg:eig:large}
Assume $M$ connected. Then, for every $m\in \N$, there is a metric $g$ on $M$ such that the Yamabe 
operator $P_{1,g}$ 
has at least $m$ negative eigenvalues counted with multiplicity. 
\end{theorem}
\begin{proof}
By a result of Lohkamp~\cite[Theorem 2]{Lo96}, given $\lambda>0$, there is a metric $g$ on $M$ 
such that
\begin{itemize}
\item[(i)] The $m$ first positive eigenvalues of the Laplacian 
$\Delta_{g}$ counted with 
multiplicity are equal to $\lambda$. 
\item[(ii)] The volume of $(M,g)$ is equal to $1$.
\item[(iii)] The Ricci curvature of $g$ is $\leq -m^2$.
\end{itemize}
The condition (iii) implies that $R_{g}\leq -nm^2$. Combining this with (ii) shows that, 
for all $u \in 
C^{\infty}(M)$, we have
\begin{align}
    \acou{P_{1,g}u}{u}&=\acou{\Delta_{g}u}{u}+ \frac{(n-2)}{4(n-1)}\int_{M}R_{g}(x)|u(x)|^2 
v_{g}(x) \nonumber \\ 
    & \leq 
    \acou{\Delta_{g}u}{u}-\frac{(n-2)}{4(n-1)}nm^{2}\|u\|^{2}. \label{eq:neg-eigv.bound-acou-Yamabe}
\end{align}

Assume $\lambda<\frac{(n-2)}{4(n-1)}nm^{2}$ and denote by $E$ the
eigenspace of $\Delta_{g}$ associated to $\lambda$. Notice that $E$ is
a subspace of $C^{\infty}(M)$ and has dimension $\ell \geq
m$. Moreover, if $u$ a unit vector in $E$,
then~(\ref{eq:neg-eigv.bound-acou-Yamabe}) shows that
$\acou{P_{1,g}u}{u}\leq
\lambda-\frac{(n-2)}{4(n-1)}nm^{2}<0$. Combining this with the min-max
principle~(\ref{eq:min-max}) we see that $\lambda_{m}(P_{1,g})\leq
\lambda_{\ell}(P_{1,g})<0$.  Thus, $P_{1,g}$ has at least $m$ negative
eigenvalues counted with multiplicity. The proof is complete.
\end{proof}

More explicit construction of metrics with an arbitrarily large number
of eigenvalues can be given in the case of a product with a hyperbolic
surface.

 Assume $n\geq 4$ and let $(N^{n-2},g_{1})$ be a compact Riemannian
 manifold and $(\Sigma,g_{2})$ a hyperbolic surface of genus~$\geq 2$.
 Given $t>0$ we equip the product $M:=N\times \Sigma$ with the product
 metric $g_{t}:=g_{1}\otimes 1+ 1\otimes t^{-1}g_{2}$.

\begin{proposition}\label{prop:negative.Yamabe.product-hyperbolic-surface}
For any $m \in \N$, we can choose $g_{2}$ and $t$ such that the Yamabe operator 
$P_{1,g_{t}}$ has at least 
    $m$ negative eigenvalues. 
\end{proposition}
\begin{proof}
    The scalar curvature of $g_{t}$ is $R_{g_{t}}=R_{g_1}-2t$, so the Yamabe operator on $(M,g_{t})$ is
    \begin{equation*}
        P_{1,g_{t}}=\Delta_{g_{1}}\otimes 1 + t(1\otimes \Delta_{g_{2}})+\frac{n-2}{4(n-1)}(R_{g_{1}}-2t), 
    \end{equation*}where $\Delta_{g_{1}}$ (resp., $\Delta_{g_{2}}$) is the Laplacian on $N$ 
(resp., $\Sigma$).
    
    Let $\lambda$ be an eigenvalue of  $\Delta_{g_{2}}$ and let $u$ be an associated eigenfunction. We have
\[P_{1,g_{t}}u=t\lambda u+ \frac{n-2}{4(n-1)}(R_{g_{1}}-2t)u.\] 
Set $\mu:=\frac{n-2}{4(n-1)}\sup_{x\in 
N}R_{g_{1}}(x)$. Then 
\begin{equation}
    \acou{P_{1,g_{t}}u}{u} \leq \left(t\lambda +\mu-\frac{n-2}{2(n-1)}t \right)\acou{u}{u} =
t\left ( \lambda 
    -\frac{n-2}{2(n-1)}+t^{-1}\mu \right) \acou{u}{u}.
    \label{eq:Yamabe-NSigma-lambda}
\end{equation}

Let $m \in \N$. Observe that
$\frac{n-2}{2(n-1)}=\frac{1}{2}-\frac{1}{2(n-1)}>\frac{1}{4}$ since
$n\geq 4$. Therefore, we may choose $t$ large enough so that $
\frac{n-2}{2(n-1)}-t^{-1}\mu >\frac{1}{4}$. Then a result of
Buser~\cite[Theorem 4]{Buser} ensures us that we can choose the metric
$g_{2}$ so that $\Delta_{g_{2}}$ has at least $m$ eigenvalues
$\lambda<\frac{n-2}{2(n-1)}- t^{-1}\mu$. If $u$ is an eigenfunction of
$u$ associated to such an eigenvalue,
then~(\ref{eq:Yamabe-NSigma-lambda}) shows that
$\acou{Pu}{u}<0$. Therefore, the quadratic form defined by
$P_{1,g_{t}}$ is negative definite on an $m$-dimensional subspace of
$C^{\infty}(M)$.  Applying the min-max principle~(\ref{eq:min-max})
then shows that $\lambda_{m}(P_{1,g_{t})}<0$, i.e., $P_{1,g_{t}}$ has
at least $m$ negative eigenvalues. The proof is complete.
\end{proof}

Next, we construct explicit examples of products of hyperbolic manifolds for which various  
higher order GJMS operators have an arbitrarily large number of negative 
eigenvalues upon varying the metric.

Let $(N^{n-2},g_{1})$ be a hyperbolic manifold and $(\Sigma^{2}, g_{2})$ a hyperbolic 
surface of 
genus~$\geq 2$. 
We equip the product manifold $M:=N\times \Sigma$ with the product metric $g:=g_{1}\otimes 
1+1\otimes g_{2}$. Notice that $(M,g)$ is an Einstein manifold and $\op{Ric}=-g$. 

\begin{theorem}\label{CriticalGJMS:hyperbolic}
For every $m\in \N$, we can choose the (hyperbolic) metric $g_{2}$ on
$\Sigma$ so that the GJMS operator $P_{k,g}$ has at least $m$ negative
eigenvalues for all odd integers $k$~$\leq \frac{n-1}{2}$.

If we further assume that $n=4l$ or $n=4l+1$ for some $l \in \N$, then the same conclusion 
holds for all integers 
$k\geq \frac{n}{2}$. 
\end{theorem}
\begin{remark}
    As $(M,g)$ is Einstein, the GJMS operator $P_{k}$ family extends
    to a family of canonical Laplacian power type operators defined for all
    integers $k >\frac{n}{2}$ even when $n$ is even \cite{Gov06}.
\end{remark}
\begin{proof}[Proof of Theorem~\ref{CriticalGJMS:hyperbolic}]
As $\op{Ric}=-g$, Eq.~(\ref{eq:GJMS-Einstein}) gives
\begin{equation}
    P_{k,g}= \prod_{1\leq j \leq k} \left(\Delta_{g}-\mu_{j}\right), \qquad 
\mu_{j}:=\frac{1}{4(n-1)}(n+2j-2)(n-2j).
    \label{eq:GJMS-hyperbolic}
\end{equation}Notice also that the Laplacian on $M$ is $\Delta_{g}=\Delta_{g_{1}}\otimes 
1+1\otimes \Delta_{g_{2}}$, 
where $\Delta_{g_{1}}$ (resp., $\Delta_{g_{2}}$) is the Laplacian on $N$ (resp., $\Sigma$).

Let $\lambda$ be an eigenvalue of  $\Delta_{g_{2}}$ and let $u$ be an eigenfunction associated 
to $\lambda$. If we 
regard $u$ as a function on $M$, then $\Delta_{g}u=\Delta_{g_{2}}u=\lambda u$. Combining this
 with~(\ref{eq:GJMS-hyperbolic}) we then see 
that $u$ is an eigenvector of $P_{k,g}$ with eigenvalue 
  \begin{equation}
      \Lambda_{k}:= \prod_{1\leq j \leq k}\left(\lambda-\mu_{j}\right). 
      \label{eq:eigenvalue-GJMS-hyperbolic}
  \end{equation}
Observe that $4(n-1)\mu_{j}=(n-1)^{2}-(2j-1)^{2}$, so $\mu_{j}>\mu_{j+1}$ for $j \geq \frac{1}{2}$. 
Moreover, 
$\mu_{\frac{n}{2}}=0$ and $\mu_{\frac{n-1}{2}}=\frac{2n-3}{4(n-1)}=\frac{1}{2}-
\frac{1}{4(n-1)}>\frac{1}{4}$. 
Incidentally, $\mu_{j}\geq 0$ when $j\geq \frac{n}{2}$. 

Let $m \in \N$. As $\mu_{\frac{n-1}{2}}>\frac{1}{4}$ and $\Sigma$ has genus~$\geq 2$, appealing again to~\cite[Theorem 4]{Buser} we can 
find a hyperbolic metric $g_{2}$ such that the Laplacian $\Delta_{g_{2}}$ has at least $m$ eigenvalues 
contained in $(0,\mu_{\frac{n-1}{2}})$. Let $\lambda$ be such an eigenvalue and assume that $k$ is an odd integer~$\leq 
\frac{n-1}{2}$. Then $\lambda-\mu_{j}\leq \lambda -\mu_{k}\leq  \lambda-\mu_{\frac{n-1}{2}}<0$ for $j=1,\cdots, k$, and 
so the eigenvalue $\Lambda_{k}$ in~(\ref{eq:eigenvalue-GJMS-hyperbolic}) is the product of $k$ negative numbers. As $k$ is odd, it follows that 
$\Lambda_{k}$ is a negative eigenvalue of $P_{k,g}$. This enables us to produce $m$ negative eigenvalues for this 
operator. 

Assume further that $m=4l$ or $m=4l+1$ for some $l \in \N$ and suppose that $k$ is an integer~$\geq \frac{n}{2}$. Then 
the integral part $k_{0}=\left[\frac{n-2}{2}\right]$ is an odd number and we can write 
\begin{equation*}
    \Lambda_{k}=\Lambda_{k}' \Lambda_{k_{0}} \qquad \text{with} \quad \Lambda_{k}':= \prod_{k_{0}+1\leq j \leq 
    k}\left(\lambda-\mu_{j}\right). 
\end{equation*}
Notice that $\Lambda_{k_{0}}<0$ since $k_{0}$ is an odd integer~$\leq \frac{n-1}{2}$. Moreover, as 
$k_{0}+1\geq \frac{n}{2}$ we see that $\mu_{j}\leq 0$ for all $j \geq k_{0}+1$, and hence $\Lambda_{k}'$ is a positive 
number. It then follows that $\Lambda_{k}$ is a negative eigenvalue of $P_{k,g}$, and so this operator has $m$ 
negative eigenvalues for this operator. The proof is complete. 
\end{proof}

Finally, we derive a version of Courant's nodal domain theorem in conformal geometry. 

\begin{theorem}\label{number:nodal}
Assume that the Yamabe operator $P_{1,g}$ has $m\geq 1$ negative 
eigenvalues.   Then any null eigenfunction of $P_{1,g}$ has at most
$m+1$ nodal domains.
\end{theorem}
\begin{proof}
By Proposition~\ref{prop:nodal-sets-Pk} the nodal domains of $P_{1,g}$ are conformal invariants. Therefore, without any loss of 
generality we may assume that the scalar curvature $R_{g}$ is constant. Then the eigenvalues of $P_{1,g}$ are obtained 
by adding $c=\frac{n-2}{4(n-1)}R_{g}$ to the eigenvalues of the Laplacian $\Delta_{g}$ and the corresponding 
eigenspaces agrees. 

Let $u \in \ker P_{1,g}$. By assumption $P_{1,g}$ has $m$ negative eigenvalues, and so the eigenvalue $\lambda=0$ is the $j$-th 
eigenvalue of $P_{1,g}$ for some $j\geq m$. It then follows that $u$ is an eigenfunction of $\Delta_{g}$ for its $j$-th 
eigenvalue. Applying Courant's nodal domain theorem~\cite{Ch:ENS, Courant} we then see that $u$ has at most $m+1$ nodal domains.  The proof is complete.
\end{proof}

\section {Curvature prescription problems}\label{section:scalar:sign}
In this section, we look at some constraints on curvature
prescription. Further results are given in Appendix. The problem of prescribing the curvature (Gaussian or scalar) of a given compact 
manifold is very classical and is 
known as the {\em Kazdan-Warner problem} (see~\cite{Au, B, KW} and the references 
therein).  The extension of this question to Branson's $Q$-curvature is an important impetus for the developement of 
various mathematical ideas (see, e.g., \cite{BaFaRe,Brendle,CGY,CY,DM,DR,MalStr,Nd:CQMAD}). 

Throughout this section we let $(M^{n},g)$ be a compact Riemannian
manifold $(n\geq 3)$. Given $k \in \N$ (further assuming $k\leq \frac{n}{2}$ if $n$ is even), 
the problem of conformally prescribing $Q_{k}$-curvature is
that of determining which functions are the $Q_{k}$-curvature $Q_{k,\hat{g}}$ for some metric $\hat{g}$ in the conformal class 
$[g]$. In other words, we seek to characterize   the
 range $\mathcal{R}(Q_k)$ of the map,
 \begin{equation}\label{map}
 Q_k:[g]\ni \hat{g}  \longrightarrow  Q_{k,\hat{g}}\in C^\infty (M,\R ). 
 \end{equation}

Let us first make some elementary observations on $\int_M R_g u\,
dv_g$ with $u\in\ker P_{1,g}$.  As $Q_{1,g}=\frac{1}{2(n-1)}R_g$,
using elementary spectral considerations and expression \eqref{int} in
the Appendix, we obtain
\begin{prop}
 Assume that the scalar curvature $R_{g}$ is constant. Then 
     \begin{equation*}
         \int_{M}R_{g}u \, dv_{g}=0 \qquad \forall u \in \ker P_{1,g}.
     \end{equation*}
In addition, if $\hat{g}=e^{2\Upsilon}g$, $\Upsilon \in C^{\infty}(M,\R)$, is a metric in the conformal class $[g]$, 
then  
     \begin{equation*}
         \int_{M}e^{\frac{2-n}{2}\Upsilon}R_{\hat{g}}u \, dv_{\hat{g}}=0 \qquad \forall u \in \ker P_{1,g}.
     \end{equation*}
\end{prop}

Next, we consider the scalar curvature restricted to nodal domains.  The following extends to the Yamabe 
operator a formula of Sogge-Zelditch~\cite[Proposition 1]{SZ:LBHMNS} for the Laplace operator.

\begin{theorem}\label{yamabe on nodal domains}
Let $u\in \ker (P_{1,g})$ and let $\Omega$ be a nodal domain of $u$. Then, for all $v \in C^\infty(M)$,
$$\int_{\Omega} |u|\,P_{1,g}(v) \, dv_g=- \int_{\partial \Omega}
v\, \| {}^{g}\nabla u \|_g\, d\sigma_g,$$
where ${}^{g}\nabla$ is the Levi-Civita connection of $g$ and $\sigma_g$ is the surface measure of $\partial \Omega$. 
\end{theorem}
\begin{remark}
    The intersection of the critical and nodal sets of $u$ has locally finite $n-2$-Hausdorff dimension 
    (\cite{HHL:GMSSEE,HHHN:CSSEE}; see also~\cite{H:SESEE,Ch:ENS}). Therefore, $\partial \Omega$ admits a normal vector 
    almost everywhere, and hence the surface measure $d\sigma_{g}$ is well-defined. 
\end{remark}
\begin{proof}[Proof of Theorem~\ref{yamabe on nodal domains}]
Notice that $u$ has constant sign on $\Omega$. Let $\nu$ be the outward unit normal vector to the hypersurface 
$\partial \Omega$. Then $\partial_\nu u$ agrees with  $- \|{}^{g}\nabla u\|_g$ (resp., $\|{}^{g}\nabla u\|_g$) almost everywhere on 
$\partial \Omega$ in case $u$ is positive (resp., negative) on 
$\Omega$. Therefore, possibly upon replacing $u$ by $-u$, we may assume that $u$ is positive on $\Omega$. 

Let $v \in C^{\infty}(M)$. As $P_{1,g}u=0$ and the Yamabe operator agrees with the Laplacian $\Delta_{g}$ up to a 
the multiplication by a function, we have 
\begin{equation*}
  \int_{\Omega} |u|\,P_{1,g}(v) \,dv_g = \int_{\Omega}\left( uP_{1,g}v -vP_{1,g}u\right)\, dv_g =  \int_{\Omega}
  \left( u\Delta_{g}v -v\Delta_{g}u\right)\, dv_g. 
\end{equation*}
Notice that $u\Delta_{g}v -v\Delta_{g}u = u\op{div}\left({}^{g}\nabla v\right)-u\op{div}\left({}^{g}\nabla 
  v\right) =\op{div}\left( u({}^{g}\nabla v)-v ({}^{g}\nabla u)\right)$. Therefore, applying the divergence theorem for rough domains 
  (see~\cite[Section 4.5.6]{Fe:GMT}), we deduce that  the integral $\int_{\Omega} |u|\,P_{1,g}(v) \,dv_g$ is equal to
\begin{equation*}
  \int_{\Omega}\op{div}\left( u({}^{g}\nabla) v-v ({}^{g}\nabla u)\right) \, dv_{g}
  = -\int_{\partial \Omega}  \left( u\, \partial_\nu v- v\, \partial_\nu u \right) \,
d\sigma_g=- \int_{\partial \Omega}
v\, \| {}^{g}\nabla u \|_g\, d\sigma_g,
\end{equation*}where we have used the fact that $u=0$ and $\partial_\nu u= -\|{}^{g}\nabla u\|_g$ on $\partial \Omega$. The 
proof is complete.
\end{proof}

Decomposing the manifold into a disjoint union of positive nodal
domains, negative nodal domains and the nodal set of $u$,
and applying Theorem~\ref{yamabe on nodal domains} we obtain

\begin{corollary}\label{green for yamabe}
For all $u\in \ker P_{1,g}$ and $v \in C^\infty(M)$,
$$\int_{M} |u|\,P_{1,g}(v) \,dv_g=- 2\int_{\mathcal{N}(u)}
v\, \| {}^{g}\nabla u \|_g\, d\sigma_g,$$
 where $\mathcal{N}(u)$ is the nodal set of $u$.
\end{corollary}

\begin{theorem}\label{int:negative}
Let $f\in C^{\infty}(M)$ be the scalar curvature of some metric in the conformal class $[g]$. Then, there is a positive function 
$\omega \in C^{\infty}(M)$, such that, for any $u\in \ker (P_{1,g})$ and any nodal domain $\Omega$ of $u$, 
$$\int_\Omega  f|u|  \omega \, dv_{g} <0.$$
\end{theorem}
\begin{proof}
By assumption $f=R_{\hat{g}}$ for some metric $\hat{g}=e^{2\Upsilon}g$, $\Upsilon 
\in C^{\infty}(M,\R)$. Thus $ P_{1,\hat{g}}(1)=\frac{n-2}{4(n-1)}R_{\hat{g}}=\frac{n-2}{4(n-1)}f$. 
Let $u\in \ker (P_{1,g})$ and let $\Omega$ be nodal domain of $u$. In addition, set 
$\omega=\frac{n-2}{4(n-1)}e^{\frac{n+2}{2}\Upsilon }$ and $\hat{u}=e^{\frac{2-n}{2}\Upsilon}u$. 
Then 
\begin{equation}
     \int_{\Omega} f  |u| \omega\, dv_{g}=  \frac{n-2}{4(n-1)} \int_{\Omega}  |\hat{u}|\, f \, 
     dv_{\hat{g}}= \int_{\Omega}|\hat{u}|P_{1,\hat{g}}(1)\, dv_{\hat{g}}.
     \label{eq:intOmega|u|}
\end{equation}

As the kernel of the Yamabe operator consists of conformal densities of weight $\frac{n-2}{2}$, we see that 
  $\hat{u}$ is contained in $P_{1,\hat{g}}$ and $\Omega$ is a nodal domain for 
 $\hat{u}$.  Therefore, applying Theorem~\ref{yamabe on nodal domains} 
 to $\hat{u}$ and $v=1$ and using~(\ref{eq:intOmega|u|}) we get
 \begin{equation*}
 \int_{\Omega} \omega  |u|\, f \, dv_{g}= \int_{\Omega}|\hat{u}|P_{1,\hat{g}}(1)\, dv_{\hat{g}}=- \int_{\partial \Omega} \| {}^{\hat g}\nabla u \|_{\hat g}\, d\sigma_{\hat{g}}.
 \end{equation*}
As the intersection of the critical and nodal sets of $u$ has locally finite $n-2$-Hausdorff 
dimension~(\cite{HHL:GMSSEE,HHHN:CSSEE}; see also~\cite{H:SESEE,Ch:ENS}), the integral $ \int_{\partial \Omega} \| {}^{\hat 
g}\nabla u \|_{\hat g}\, d\sigma_{\hat{g}}$ must be positive, and hence $ \int_{\Omega} \omega  |u|\, f \, dv_{g}<0$.  
This proves the result. 
\end{proof}

Theorem~\ref{int:negative} seems to be new.  We remark that when $\dim\ker(P_{1,g})\geq 2$ 
this gives infinitely many constraints on $R_{\hat g}$.


\begin{corollary}\label{scalar:sign}
Let $u\in \ker (P_{1,g})$ and let $\Omega$ be nodal domain of $u$. Then, for any metric 
$\hat g$ in the conformal class 
$[g]$, the scalar curvature $R_{\hat g}$ cannot be everywhere positive on $\Omega$.
\end{corollary}

Let $u\in \ker (P_{1,g})$ and let $\Omega$ be nodal domain of $u$. Given any metric 
$\hat{g}=e^{2\Upsilon} g$, $\Upsilon \in 
C^{\infty}(M,\R)$, in the conformal class $[g]$ we define
$$
T(u,\Omega, \hat g):= - \frac{4(n-1)}{n-2} 
 \int_{\partial \Omega}  e^{\frac{2-n}{2}\Upsilon }\| {}^{\hat{g}}\nabla  \hat{u} \|_{\hat g}\; 
d\sigma_{\hat g},
$$
where we have set $\hat{u}=e^{\frac{2-n}{2}\Upsilon}u$.

\begin{proposition}\label{int:fixed}
For all metrics $\hat{g}$ in the conformal class $[g]$,
$$
T(u,\Omega, \hat g)= \int_\Omega |u|R_g\, dv_g.
$$
\end{proposition}
\begin{proof}
Let $\hat{g}=e^{2\Upsilon} g$, $\Upsilon \in C^{\infty}(M,\R)$,  be a metric in the 
conformal class $[g]$. Set $\hat{u}=e^{\frac{2-n}{2}\Upsilon}u$ and 
$v=e^{\frac{2-n}2\Upsilon }$. 
As pointed out in the proof of Theorem~\ref{int:negative}, $\hat{u}$  
lies in $\ker P_{1,\hat{g}}$ and $\Omega$ is a nodal domain. Applying 
Theorem~\ref{yamabe on nodal domains} 
 to $\hat{u}$ and $v$ then gives
 \begin{equation*}
    \frac{n-2}{4(n-1)}T(u,\Omega, \hat g)= - \int_{\partial \Omega} v \|
{}^{\hat{g}}\nabla \hat u \|_{\hat g}\; d\sigma_{\hat g} = \int_{\Omega} |\hat u|\,P_{1,\hat g}v\,dv_{\hat g}.
 \end{equation*}
 As $P_{1,\hat{g}}v= e^{-\frac{n+2}{2}\Upsilon}P\left( e^{\frac{n-2}{2}\Upsilon} \cdot e^{\frac{2-n}2\Upsilon 
 }\right)=e^{-\frac{n+2}{2}\Upsilon}P_{1,g}(1)= \frac{n-2}{4(n-1)}e^{-\frac{n+2}{2}\Upsilon}R_{g}$, we get
 \begin{equation*}
 T(u,\Omega, \hat g)= \int_{\Omega} |\hat u|  e^{-\frac{n+2}{2}\Upsilon}R_{g} \, dv_{\hat{g}}=  \int_\Omega 
 e^{\frac{2-n}{2}\Upsilon}  |u|  e^{-\frac{n+2}{2}\Upsilon}R_{g}\, e^{n\Upsilon}\, dv_g=
   \int_\Omega |u|R_g\, dv_g. 
 \end{equation*}
 The result is proved. 
\end{proof}

Proposition~\ref{int:fixed} provides us with some conserved quantities for the conformal class.
In particular, if $R_g$ is constant,  then we obtain
$$
T(u,\Omega, \hat g)=R_g||u||_{L^1(\Omega)}.
$$

Finally, let $k \in \N$ and further assume $k\leq  \frac{n}{2}-1$ when $n$ is even. We look at 
conformal classes containing a metric for which the $Q_{k}$-curvature is zero.  

\begin{proposition}\label{Pk:kernel:const}
    The following are equivalent:
    \begin{enumerate}
        \item  The kernel of $P_{k,g}$ contains a nowhere vanishing eigenfunction. 
        
        \item There is a metric $\hat{g}$ in the conformal class $[g]$ such that $P_{k,\hat{g}}(1)=0$.
    
        \item  There is a metric $\hat{g}$ in the conformal class $[g]$ such that $Q_{k,\hat{g}}$ 
is identically zero. 
    \end{enumerate}
\end{proposition}
\begin{proof}
As by definition $Q_{k,\hat{g}}=\frac{2}{n-2k}P_{k,\hat{g}}(1)$, the
equivalence of~(2) and~(3) is immediate.  Furthermore, the conformal
invariance of $\ker P_{k,g}$ as a space of conformal densities of
weight $k-\frac{n}{2}\neq 0$ implies that $\ker P_{k,g}$ contains a
nowhere vanishing function if and only if there is a metric $\hat{g}$
in the conformal class $[g]$ such that $\ker P_{k,\hat{g}}(1)=0$.
This proves the equivalence of (1) and~(3) and completes the proof.
\end{proof}

\begin{remark}
In the recent paper~\cite{FR} A.\ Fardoun and R.\ Regbaoui study the $Q$-curvature prescription problem on even-dimensional conformal 
classes for which the kernel of the critical GJMS operator 
is nontrivial kernel (i.e., it contains non-constant functions).  In particular, they give sufficient conditions 
for the convergence of the $Q$-curvature flow in terms of nodal domains of null-eigenfunctions of the critical GJMS operator.  
\end{remark}

\section{The Yamabe and Paneitz Operators on Heisenberg Manifolds} \label{section: heisenberg}
In this section, we explicitly compute  the eigenvalues and nodal sets of the Yamabe and Paneitz operators on Heisenberg 
manifolds. 

\subsection{The setup}
Let $\bH_{d}$ be the $(2d+1)$-dimensional Heisenberg group, i.e., the 2-nilpotent subgroup of $\op{GL}_{d+2}(\R)$
of unipotent matrices. Thus, any $A\in \bH_{d}$ is of the form,
\begin{equation*}
A=    \begin{pmatrix}
1 & x & t\\
0 & 1 & y^{T}\\
0&0&1
\end{pmatrix}, \qquad x,y\in \R^{d}, \quad t\in \R.
\end{equation*}
We shall use coordinates $x=(x_{1},\cdots x_{d})$, $y=(y_{1},\cdots, y_{d})$ and $t$ as above to represent an element of
$\bH_{d}$.

Let $r=(r_{1},\cdots, r_{d})\in \Z^{d}$ be such that $r_{j}|r_{j+1}$ for $j=1,\cdots, d-1$ and consider the lattice
subgroup,
\begin{equation*}
    \Gamma_{r}=\left\{  \begin{pmatrix}
1 & x & t\\
0 & 1 & y^{T}\\
0&0&1
\end{pmatrix}; \ x\in \Z^{d}, \ y\in r_{1}\Z \times \cdots \times r_{d}\Z, \ t \in \Z \right\}.
\end{equation*}In addition, consider the quotient manifold,
\begin{equation*}
    M:=\Gamma_{r} \backslash \bH_{d}.
\end{equation*}This is a compact manifold with fundamental group $\bH_{d}$.
Moreover, a fundamental domain for this quotient is \[D=[0,1)^{d}\times
[0,r_{1})\times \cdots \times[0,r_{d})\times [0,1).\]

A tangent frame of $\bH_{d}$ is provided by the left-invariant vector fields, 
\begin{equation*}
    X_{j}=\frac{\bd}{\bd x_{j}},
\qquad Y_{j}=\frac{\bd}{\bd y_{j}}+x_{j}\frac{\bd}{\bd t},
\qquad T=\frac{\bd}{\bd t},
\end{equation*}where $j$ ranges over $1, \cdots, d$. The
standard contact form on $\cH_{d}$ is the left-invariant
1-form given by
\begin{equation*}
    \theta:=dt -\sum_{1\leq j \leq d}x_{j}dy_{j}.
\end{equation*}Notice that the 1-forms $dt$ and $dx_{j}$ and $dy_{j}$ are
left-invariant too.

Let $s>0$. We endow $\bH_{d}$ with the left-invariant metric,
\begin{equation}
    g_{s}:= \sum_{1\leq j \leq d}dx_{j}
\otimes dx_{j}+\sum_{1\leq j \leq d}s^{-2}dy_{j}\otimes dy_{j}+s^{2d} \theta
    \otimes \theta.
    \label{eq:Heisenberg.metric-gs}
\end{equation}This is the type of left-invariant Riemannian metrics
considered in~\cite{GW}. Notice that the volume of $M$ with respect to $g_{s}$ is independent of $s$ and is equal to
\begin{equation*}
    \left| \Gamma_{r}\right|:= r_{1}\cdots r_{d}.
\end{equation*}
Observe that
$\{X_{j},sY_{j},s^{-d}T\}$ is an orthonormal frame for this metric and $\det g_{s}=1$. Therefore, the Laplace
operator $\Delta_{g_{s}}$ on $\bH_{d}$ is given by
\begin{equation}
    \Delta_{g_{s}}=-\sum_{1\leq j \leq d}(X_{j}^{2}+s^{2}Y_{j}^{2}) -s^{-2d}T^{2}.
    \label{eq:Heisenberg.Laplacian}
\end{equation}

In addition, it follows from the results of~\cite{Jen} that the Ricci
tensor of $g_{s}$ (seen as a $(0,2)$-tensor) is given by
\begin{align}
    \op{Ric}_{g_{s}}&=-\frac{1}{2}s^{2d+2}\sum_{1\leq j \leq d} (dx_{j}\otimes dx_{j}+s^{-2}dy_{j}\otimes
    dy_{j})+\frac{d}{2}s^{4d+2}\theta \otimes \theta  \nonumber \\
    &=-\frac{1}{2}s^{2d+2}g+\frac{(d+1)}{2}s^{4d+2}\theta \otimes \theta. \label{eq:Ricci-Heisenberg}
\end{align}
We then get the following formula for the scalar curvature,
\begin{equation}
        R_{g_{s}}= -\frac{d}{2}s^{2d+2}.
        \label{eq:Scalar-Heisenberg}
\end{equation}

\subsection{Spectral resolution of the Yamabe operator}
As all the objects above are left-invariant, they descend to $M=\Gamma \backslash \bH_d$. In particular, the Laplacian $\Delta_{g_{s}}$ on $M$
is given by~(\ref{eq:Heisenberg.Laplacian}), where the vector fields $X_{j}$, $Y_{j}$ and $T$ are meant as vector fields on $M$.
Therefore, combining this with the above formula for the scalar curvature we obtain

\begin{proposition}\label{prop:Heisenberg.Yamabe}
    The Yamabe operator of $(M,g_{s})$ is given by
\begin{equation*}
        P_{1,g_{s}}= -\sum_{1\leq j \leq d}(X_{j}^{2}+s^{2}Y_{j}^{2}) -s^{-2d}T^{2}-\frac{2d-1}{16}s^{2d+2}.
    \end{equation*}
\end{proposition}

The spectral resolution of $\Delta_{g_{s}}$ on $M$ is intimately related to the representation theory of $\bH_{d}$.
Indeed, the right-action of $\bH_{d}$ on itself descends to a right-action on $M$, and hence the right-regular
representation descends to the unitary representation
\[\rho:\bH_{d}\longrightarrow L^{2}(M).\]
This representation can be decomposed into irreducible representations as follows.

Recall that the irreducible representations of $\bH_{d}$ are of two types:
\begin{enumerate}
    \item[(i)]   The characters
$\chi_{(\xi,\eta)}:\bH_{d}\rightarrow \mathbb{C}$, $(\xi,\eta)\in \R^{d}\times \R^{d}$, defined by
\begin{equation}
    \chi_{(\xi,\eta)}(x,y,t)=e^{2i\pi(\xi\cdot x+\eta \cdot y)}.
    \label{eq:Heisenberg.chixieta}
\end{equation}

    \item[(ii)] The infinite dimensional representations $\pi_{h}:\bH_{d}\rightarrow
\cL\left(L^{2}(\R^{d})\right)$, $h \in \R^{*}$, given by
\begin{equation}
   \left[ \pi_{h}(x,y,t)f\right](\xi):=e^{2i\pi h(t+y\cdot \xi)}f(\xi+x) \qquad \forall f \in L^{2}(\R^{d}).
   \label{eq:Heisenberg.pih}
\end{equation}
\end{enumerate}
We observe that for the characters $ \chi_{(\xi,\eta)}$ we have
\begin{equation}
    X_{j}\chi_{(\xi,\eta)}=2i\pi \xi_{j}\chi_{(\xi,\eta)}, \qquad Y_{j}\chi_{(\xi,\eta)}(Y_{j})=2i\pi \eta_{j}\chi_{(\xi,\eta)}, \qquad
    T\chi_{(\xi,\eta)}=0.
    \label{eq:Heisenberg.XYT-chi}
\end{equation}For the  representations $\pi_{h}$, we have
\begin{equation}
    d\pi_{h}(X_{j})=\frac{\bd }{\bd \xi_{j}}, \qquad d\pi_{h}(Y_{j})=2i\pi h \xi_{j}, \qquad d\pi_{h}(T)=2i\pi h.
    \label{eq:Heisenberg.XYT-pih}
\end{equation}

For $n \in \Z$, define
\begin{equation*}
    \cH_{n}:=\left\{ f \in L^{2}(M); \ f(x,y,t)=e^{2in\pi t}g(x,y)\right\}.
\end{equation*}In particular, $\cH_{0}$ is the space of functions that do not depend on the $t$-variable.

Define
\begin{gather*}
    \Lambda =\Z^{d} \times \left((r_{1}\Z)\times \cdots \times (r_{d}\Z)\right),\\
    \Lambda'=\{(\mu,\nu)\in \R^{d}\times \R^{d};\ \mu\cdot x+\nu \cdot y \in \Z \quad \forall (x,y)\in \Lambda\}.
\end{gather*}
Notice that $\Lambda$ is the lattice of $\R^{2d}$ given by the image of $\Gamma$ under the projection
$(x,y,t)\rightarrow (x,y)$. The set $\Lambda'$ is its dual lattice.

Let $n \in \Z^{*}$. We set
\begin{gather*}
    \cA_{n}:=\left\{a=(a_{1},\cdots, a_{d})\in \R^{d}; \ a_{j}\in \biggl\{0,\frac{1}{|n|},\cdots,
    \frac{|n|-1}{|n|}\biggr\}\right\},\\
   \cB:=\left\{b=(b_{1},\cdots, b_{d}); b_{j}\in \biggl\{0,\frac{1}{r_{j}},\cdots, \frac{r_{j}-1}{r_{j}}\biggr\}\right\}.
\end{gather*}
For $a\in \cA_{n}$ and $b\in \cB$, we define the operator $W_{n}^{a,b}:L^{2}(\R^{d})\rightarrow L^{2}(M)$ by
\begin{equation*}
    W_{n}^{a,b}f(x,y,t)=e^{2i\pi n t} \sum_{k \in \Z_{d}}f(x+k+a+b)e^{2i\pi n(k+a+b)\cdot y}.  
\end{equation*}This is an isometry from $L^{2}(\R^{d})$ into $\cH_{n}$ (see~\cite{Fo}).
We then let \[\cH_{n}^{a,b}:=W_{n}^{a,b}(L^{2}(\R^{d})).\]

\begin{proposition}[Brezin~\cite{Bre}]
   We have the following orthogonal decompositions,
   \begin{gather*}
       L^{2}(M)=\bigoplus_{n\in \Z}\cH_{n},\\ \cH_{0}=\bigoplus_{(\mu,\nu)\in \Lambda^{'}}\mathbb{C} \chi_{(\mu,\nu)},
       \qquad \cH_{n}=\bigoplus_{\substack{a\in \cA_{n}\\ b\in \cB}} \cH_{n}^{a,b}, \quad n\neq 0.
   \end{gather*}Moreover, each operator $W_{n}^{a,b}$ is an intertwining operator from $\pi_{n}$ to the regular representation
   $\rho$. In particular, the multiplicity of $\pi_{n}$ in $\rho$ is equal to $|n|^{d}|r_{1}\cdots
   r_{d}=|n|^{d}|\Gamma_{r}|$.
\end{proposition}

Thanks to this result the spectral analysis of $\Delta_{g_{s}}$  on $L^{2}(M)$ reduces to the spectral analysis on each
of the irreducible
subspaces $\mathbb{C} \chi_{(\mu,\nu)}$ and $\cH_{n}^{a,b}$.

Let $(\xi,\eta)\in \Lambda'$. Then from~(\ref{eq:Heisenberg.Laplacian}) and~(\ref{eq:Heisenberg.XYT-chi}) we see that
\begin{equation*}
    \Delta_{g_{s}}\chi_{(\xi,\eta)}=4\pi^{2}(|\xi|^{2}+s^{2}|\eta|^{2})\chi_{(\xi,\eta)}.
\end{equation*}That is, $\chi_{(\xi,\eta)}$ is an eigenfunction of $\Delta_{g_{s}}$ w.r.t.~the eigenvalue
$\lambda=4\pi^{2}(|\xi|^{2}+s^{2}|\eta|^{2})$.

Let $n\in \Z^{*}$. Using~(\ref{eq:Heisenberg.Laplacian}) and~(\ref{eq:Heisenberg.XYT-pih}) we get
\begin{equation*}
    d\pi_{n}(\Delta_{g_{s}})=\sum_{1\leq j \leq d}\left(- \partial_{\xi_{j}}^{2}+4n^{2}s^{2}\pi^{2}
    \xi_{j}^{2}\right)+4n^{2}s^{-2d}\pi^{2}.
\end{equation*}Under the change of variable $\eta_{j}=\sqrt{2\pi|n|s}\,\xi_{j}$ this becomes
\begin{equation*}
    2\pi |n|s \sum_{1\leq j \leq d}\left(- \partial_{\eta_{j}}^{2}+
    \eta_{j}^{2}\right)+4n^{2}s^{-2d}\pi^{2}.
\end{equation*}

Recall that on $\R$ that an orthogonal eigenbasis of $L^{2}(\R)$ for the harmonic oscillator $- \frac{d^{2}}{dv^{2}}+v^{2}$
is given by the Hermite functions,
    $h_{k}(v)$, $k \in
    \N_{0}$, such that
    \begin{equation*}
      h_{k}(v):=(-1)^{k}\frac{d^{k}}{dv^{k}}e^{-\frac{1}{2}v^{2}}, \qquad   \left(-
      \frac{d^{2}}{dv^{2}}+v^{2}\right)h_{k}(v)=(1+2k)h_{k}(v).
    \end{equation*}Notice that $e^{\frac{1}{2}v^{2}}h_{k}(v)$ is a polynomial of degree $k$.

    From all this we deduce that an orthogonal basis of eigenfunctions of $d\pi_{n}(\Delta_{g_{s}})$ is given by
    the functions,
    \begin{equation*}
        f_{\alpha}(\xi):=\prod_{1\leq j \leq d}h_{\alpha_{j}}(\sqrt{2\pi|n|s}\, \xi_{j}), \qquad \alpha \in \N_{0}^{d},
    \end{equation*}in such way that
    \begin{equation*}
        d\pi_{n}(\Delta_{g_{s}})f_{\alpha}=\left(2\pi |n|s (d+2|\alpha|)+4n^{2}s^{-2d}\pi^{2}\right)f_{\alpha}.
    \end{equation*}Notice that each eigenvalue $2\pi |n|s (d+2|\alpha|)+4n^{2}s^{-2d}\pi^{2}$ occurs with multiplicity
    $\displaystyle\binom{|\alpha|+d-1}{d-1}$.

As $W_{n}^{a,b}$ intertwines from $\pi_{n}$ to the regular representation $\rho$, we see that an orthogonal eigenbasis of $\cH_{n}^{a,b}$ for
$\Delta_{g_{s}}$ is given by the functions,
\begin{align}
    W_{n}^{a,b}f_{\alpha}(x,y,t) & =e^{2in \pi t} \sum_{k\in \Z^{d}}f_{\alpha}(x+k+a+b)e^{2in\pi (k+a+b)\cdot y} \nonumber \\
    & = e^{2in \pi t} \sum_{k\in \Z^{d}} \prod_{1\leq j \leq d}h_{\alpha_{j}}\left(\sqrt{2\pi|n|s}\,
    (x_{j}+k_{j}+a_{j}+b_{j})\right)e^{2in\pi (k_{j}+a+b)y_{j}} \label{eq:Heisenberg.wnabfalpha}\\
   &= e^{2in \pi t} \prod_{1\leq j \leq d} \left\{\sum_{k\in \Z} h_{\alpha_{j}}\left(\sqrt{2\pi|n|s}\,
    (x_{j}+k_{j}+a_{j}+b_{j})\right)e^{2in\pi (k+a+b)y_{j}}\right\}.\nonumber
\end{align}Each $W_{n}^{a,b}f_{\alpha}$ is an eigenfunction for the eigenvalue $2\pi |n|s
(d+2|\alpha|)+4n^{2}s^{-2d}\pi^{2}$. This eigenvalue has multiplicity $\displaystyle |n|^{d}r_{1}\cdots
   r_{d} \binom{|\alpha|+d-1}{d-1}$ in $\cH_{n}$.

   As it turns out, for $\alpha =0$ the function  $W_{n}^{a,b}f_{0}$ can be expressed in terms of Jacobi's theta function,
\begin{equation*}
    \vartheta(z,\tau):= \sum_{k\in \Z} e^{i\pi k^{2}\tau}e^{2i\pi k z}, \qquad z,\tau\in  \mathbb{C}, \ \Im \tau>0.
\end{equation*}If $\alpha =0$, then $h_{\alpha_{j}}(v)=h_{0}(v)=e^{-\frac{1}{2}v^{2}}$ for $j=1,\cdots, d$. Moreover,
for $u>0$ and $v,c\in \R$, we have
\begin{align*}
    \sum_{k\in \Z} h_{0}\left(\sqrt{2\pi|n|s}\,
    (u+k)\right)e^{2in\pi (k+c)v}
     &= \sum_{k\in \Z}e^{-\pi|n|s(u+k)^{2}}e^{2in\pi (k+c)v} \\
      &= e^{2in\pi cv} \sum_{k\in \Z}e^{-\pi|n|s(u^{2}+2ku+u^{2})}e^{2in\pi kv} \\
      &= e^{-\pi|n|su^{2}} e^{2in\pi cv}\sum_{k\in \Z}e^{-\pi|n|s u^{2}}e^{2in\pi k(v\pm i su)} \\
      &= e^{-\pi|n|su^{2}} e^{2in\pi cv}\vartheta(v\pm isu, i|n|s),
\end{align*}where $\pm$ is the sign of $n$. Applying this equality to $u=x_{j}+a_{j}+b_{j}$, $v=y_{j}$ and
$c=a_{j}+b_{j}$ and using~(\ref{eq:Heisenberg.wnabfalpha}) we get

\begin{align}
    &W_{n}^{a,b}f_{0}(x,y,t) = \nonumber\\ 
    &\quad =e^{2in \pi t} \prod_{1\leq j \leq d} \left\{ e^{-\pi|n|s(x_{j}+a_{j}+b_{j})^{2}}
    e^{2in\pi (a_{j}+b_{j})v}\vartheta\left(y_{j}\pm is(x_{j}+a_{j}+b_{j}), i|n|s\right) \right\} \nonumber\\
     & \quad= e^{2in \pi t} e^{-\pi |n| s \, |x+a+b|^{2}}e^{2in\pi (a+b)\cdot v}\prod_{1\leq j \leq d} \vartheta\left(y_{j}\pm
    is(x_{j}+a_{j}+b_{j}), i|n|s\right). \label{eq:Heisenberg.wnabf0}
\end{align}

   Combining  the previous discussion with Proposition~\ref{prop:Heisenberg.Yamabe} we obtain

\begin{proposition}[Compare~{\cite[Theorem 3.3]{GW}}]\label{Yamabe:nil} 
Assume $M=\Gamma_{r}\backslash \bH_{d}$ is equipped with the metric $g_{s}$ given
    by~(\ref{eq:Heisenberg.metric-gs}).
    \begin{enumerate}
        \item  An orthogonal eigenbasis of $L^{2}(M)$ for the Yamabe operator $P_{1,g_{s}}$ is given by the join,
        \begin{equation}
           \left\{  \chi_{(\xi,\eta)};\ (\xi,\eta)\in \Lambda'\right\} \bigcup \left\{W_{n}^{a,b}f_{\alpha};\
           (n,a,b,\alpha)\in\Z^{*}\times \cA_{n}\times \cB\times
            \N_{0}^{d}\right\}.
            \label{eq:Heisenberg.eigenfunctions}
\end{equation}The characters $ \chi_{(\xi,\eta)}$ are  given by~(\ref{eq:Heisenberg.chixieta}). 
The functions
$W_{n}^{a,b}f_{\alpha}$ are given by~(\ref{eq:Heisenberg.wnabfalpha}), which reduces 
to~(\ref{eq:Heisenberg.wnabf0}) when $\alpha=0$.

        \item  Each character $\chi_{(\xi,\eta)}$ is an eigenfunction of $P_{1,g_{s}}$ 
with eigenvalue
        \begin{equation*}
         \lambda(\xi,\eta):= 4\pi^{2}\left(|\xi|^{2}+s^{2}|\eta|^{2}\right)-\frac{2d-1}{16}s^{2d+2}.
        \end{equation*}

        \item  Each function $W_{n}^{a,b}f_{\alpha}$ is an eigenfunction of $P_{1,g_{s}}$ 
with eigenvalue
        \begin{equation*}
         \mu(n,|\alpha|):=  2 \pi |n|s (d+2|\alpha|)+4n^{2}s^{-2d}\pi^{2}-\frac{2d-1}{16}s^{2d+2}.
        \end{equation*}This eigenvalue  has multiplicity $\displaystyle |n|^{d}|\Gamma_{r}|
        \binom{|\alpha|+d-1}{d-1}$ in $\cH_{n}$.
    \end{enumerate}
 \end{proposition}

 \begin{remark}[See also~{\cite[Theorem 3.3]{GW}}]\label{rmk:spectrum-Laplacian-Heisenberg}
 The above considerations also provides us with a spectral resolution of the Laplacian $\Delta_{g_{s}}$. More 
 precisely, we see that 
  \begin{enumerate}
      \item[(i)]  Each $\chi_{(\xi,\eta)}$ is an eigenfunction of $\Delta_{g_{s}}$ 
with eigenvalue
        \begin{equation}
         \lambda_{0}(\xi,\eta):= 4\pi^{2}\left(|\xi|^{2}+s^{2}|\eta|^{2}\right).
         \label{eq:spectrum-Laplacian-Heisenberg-l0}
        \end{equation}
  
      \item[(ii)] Each function $W_{n}^{a,b}f_{\alpha}$ is an eigenfunction of $P_{1,g_{s}}$ 
with eigenvalue
        \begin{equation}
         \mu_{0}(n,|\alpha|):=  2 \pi |n|s (d+2|\alpha|)+4n^{2}s^{-2d}\pi^{2}.
           \label{eq:spectrum-Laplacian-Heisenberg-m0}
        \end{equation}
  \end{enumerate}
 \end{remark}


\subsection{Nodal sets and negative eigenvalues of $P_{1,g_{s}}$} 

We observe that an eigenfunction $W_{n}^{a,b}f_{\alpha}$ lies in the nullspace of $P_{1,g_{s}}$ if and only if
\begin{equation}
    2 \pi |n|s (d+2|\alpha|)+4n^{2}s^{-2d}\pi^{2}-\frac{2d-1}{16}s^{2d+2}=0.
    \label{eq:condition-null-eigenvalue-Paneitz-Heisenberg}
\end{equation}If we multiply both sides of the equation by $-s^{2d}$ and set $v=s^{2d+1}$, then this equation
becomes the quadratic equation $ \frac{2d-1}{16}v^{2}-2 \pi |n| (d+2|\alpha|)v-4n^{2}\pi^{2}=0$, 
whose unique positive root is
 $ v=\frac{8\pi |n|}{2d-1}\left(2(d+2|\alpha|)+\sqrt{4(d+2|\alpha|)^{2}+2d-1}\right)$.
Therefore, $W_{n}^{a,b}f_{\alpha}$ lies in the nullspace of $P_{1,g_{s}}$ if and only if
\begin{equation}
  s^{2d+1}= \frac{8\pi |n|}{2d-1}\left(2(d+2|\alpha|)+\sqrt{4(d+2|\alpha|)^{2}+2d-1}\right).
  \label{eq:Heisenberg.zero-eigenvalue}
\end{equation}

Let us give a lower bound for $\nu_{1}(g_{s})$, i.e., the number of negative eigenvalues of $P_{1,g_{s}}$. 
An eigenvalue $\mu(n,|\alpha|)=  2 \pi |n|s (d+2|\alpha|)+4n^{2}s^{-2d}\pi^{2}-\frac{2d-1}{16}s^{2d+2}$
is negative if and only if
\begin{equation*}
  s^{2d+1}> \frac{8\pi |n|}{2d-1}\left(2(d+2|\alpha|)+
\sqrt{4(d+2|\alpha|)^{2}+2d-1}\right).
\end{equation*}In particular, for every integer $n$ such that  $\frac{8\pi |n|}{2d-1}\left(2d+\sqrt{4d^{2}+2d-1}\right)<s^{2d+1}$,
the eigenvalue $\mu(n,0)$ is negative. Moreover, such an eigenvalue occurs with multiplicity  $\displaystyle 
|n|^{d}|\Gamma_{r}|$. Therefore, we obtain

\begin{proposition}\label{prop:neg-Yamabe-Heisenberg}
There is a constant $c_{d}>0$ depending on $d$, but not on the sequence $r=(r_{1},\cdots, r_{d})$, such that
\begin{equation*}
    \nu_{1}(g_{s})\geq c_{d}|\Gamma_{r}|s^{2d+2} \qquad \forall s>0. 
\end{equation*} 
In particular, for every integer $m\in \N$, the Yamabe operator $P_{1,g_{s}}$ has at least $m$ negative 
eigenvalues as soon as $s$ is large enough. 
\end{proposition}

Suppose now that $s^{2d+1}= \frac{8\pi}{2d-1}\left(2d+\sqrt{4d^{2}+2d-1}\right)$, 
i.e.,~(\ref{eq:Heisenberg.zero-eigenvalue}) holds for $n=\pm 1$
and $\alpha=0$. Notice that $\cA_{\pm 1}=\{0\}$, so $\ker P_{1,g_{s}}\cup \cH_{0}^{\perp}$ is spanned
by the functions $W_{\pm 1}^{0,b}f_{0}$ given by~(\ref{eq:Heisenberg.wnabf0}) where $b$ ranges over $\cB$. 
Furthermore, the transcendence of $\pi$ implies that if
$s^{2d+1}= \frac{8\pi}{2d-1}\left(2d+\sqrt{4d^{2}+2d-1}\right)$, then no
element $(\xi,\eta)\in \Lambda'$ satisfies $\lambda(\xi,\eta)=|\xi|^{2}+s^{2}|\eta|^{2}-\frac{2d-1}{16}s^{2d+2}=0$,
i.e., no character $\chi_{(\xi,\eta)}$ is contained in $\ker P_{1,g_{s}}$. Thus the functions $W_{\pm 1}^{0,b}f_{0}$,
$b \in \cB$, form an orthogonal eigenbasis of $\ker P_{1,g_{s}}$.

Let us now look at the nodal sets of the eigenfunctions $W_{\pm 1}^{0,b}f_{0}$. 
It follows from~(\ref{eq:Heisenberg.wnabf0}) that $W_{\pm
1}^{0,b}f_{0}(x,y,t)=0$ if and only if
\begin{equation}
    \vartheta\left(y_{j}\pm is(x_{j}+b_{j}), is\right)=0 \qquad \text{for some $j\in \{1,\cdots,d\}$}.
    \label{eq:Heisenberg.nodal-set-theta-condition}
\end{equation}
Moreover, by Jacobi's triple product formula,
\begin{equation*}
         \vartheta(z,is) = \prod_{m=1}^{\infty}(1-e^{-2m\pi s})(1+e^{2i\pi z}e^{-(2m-1)\pi s})(1+e^{-2i\pi
       z}e^{-(2m-1)\pi s}).
    \end{equation*}
Thus, for $z=v+isu$, $u,v\in \R$, we see that $\vartheta(z,is)=0$ if and only if there 
is $m\in \Z$ such that
$e^{-(2m-1)\pi s}=-  e^{2i \pi z} =e^{-2\pi s u+i\pi (2v+1)}$, that is, $u$ and $v$ are contained in 
$\frac{1}{2}+\Z$.
Applying this to $u=x_{j}+b_{j}$ and $v= \pm y_{j}$ with $(x_{j},y_{j})\in [0,1)\times [0,r_{j})$ 
and $b_{j}\in \{0,r_{j}^{-1}, \cdots,
1-r_{l}^{-1}\}$ it not hard to deduce that
\begin{equation*}
   \vartheta\left(y_{j}\pm is(x_{j}+b_{j}), is\right)=0 \Longleftrightarrow \left\{
   \begin{array}{l}
     x_{j}=\pm
(b_{j}-\frac{1}{2})-\left[\pm (b_{j}-\frac{1}{2})\right], \\
     y_{j}\in \left\{\frac{1}{2},\frac{3}{2},\cdots,
r_{j}-\frac{1}{2}\right\},
   \end{array}\right.
\end{equation*}where $\left [\cdot \right]$ is the floor function. (Here $x_{j}=\pm
(b_{j}-\frac{1}{2})-\left[\pm (b_{j}-\frac{1}{2})\right]$ is the only element of $[0,1)$ such 
that $\pm (x_{j}+b_{j})$ is
a half-integer.) Combining with~(\ref{eq:Heisenberg.nodal-set-theta-condition}) enables us to 
get the nodal set of $W_{\pm 1}^{0,b}f_{0}$.

Summarizing the previous discussion, we have proved

\begin{proposition}\label{prop:negative-eigenvalues-Yamabe-Heisenberg}
  Let $s$ be the $(2d+1)$-th root of  
$\frac{8\pi}{2d-1}\left(2d+\sqrt{4d^{2}+2d-1}\right)$.
  \begin{enumerate}
      
\item  The functions $W_{\pm 1}^{0,b}f_{0}$, $b\in \cB$, form an 
orthogonal basis of $\ker P_{1,g_{s}}$.

\item  The nodal set of $W_{\pm 1}^{0,b}f_{0}$ is given by the join,
      
\begin{equation*}
          \bigcup_{\substack {1\leq j \leq d\\ 1\leq l \leq r_{j}}}
\left\{(x,y,t)\in M;  x_{j}=\pm
(b_{j}-\frac{1}{2})-\left[\pm (b_{j}-\frac{1}{2})\right] ,
\  y_{j}=l-\frac{1}{2}\right\}.
      \end{equation*}
  \end{enumerate}
\end{proposition}
\begin{remark}
The nodal sets of the eigenfunctions 
$W_{\pm 1}^{0,b}f_{0}$ are submanifolds of codimension~2 in $M$.    
\end{remark}


Null eigenvectors and negative eigenvalues can also occur from the characters $\chi_{(\xi,\eta)}$. 
To simplify the discussion we assume that $d=1$ and $r_{1}=1$.  

The eigenfunction $\chi_{(\xi,\eta)}$ is in the kernel of $P_{1,g_{s}}$ if 
and only if $\lambda(\xi,\eta)=0$ which, according to Proposition \ref{Yamabe:nil}, is equivalent 
to 
\begin{equation}\label{zero:eig:torus}
4\pi^2(|\xi|^{2}+s^{2}|\eta|^{2})=\frac{2d-1}{16}s^{2d+2}
\end{equation}

Although $\chi_{(\xi,\eta)}$ takes values in $\mathbb{S}^{1}$, we observe that $\chi_{(\xi,\eta)}$ is a null 
eigenvector if and only if so is $\chi_{(-\xi,-\eta)}=\overline{\chi_{(\xi,\eta)}}$. Therefore, nodal sets to 
consider are $\cN(\Re\chi_{\xi,\eta})$ and 
$\cN(\Im\chi_{\xi,\eta})$.  Notice they are of the 
form $\cN(x,y)\times [0,1]$ where the last factor corresponds to the $t$ coordinate.  

As $r_{1}=1$ we see that $\xi$ and $\eta$ are in $\zed$ in \eqref{zero:eig:torus}.  Accordingly, the 
values of $s$ for which a function $\chi_{(\xi,\eta)}$ 
lies in the kernel of $P_{1,g_s}$ are given by 
\begin{equation}\label{square:root}
s^2\in\left\{32\pi^2(\eta^2+\sqrt{\eta^4+\xi^4/(16\pi^2)}):\xi,\eta\in\zed\right\}.  
\end{equation}
Here the dimension of the kernel is equal to the number of the pairs $(\xi,\eta)$ for which 
equation \eqref{zero:eig:torus} has solutions for a given $s\in\reals$.  An elementary 
calculation shows that $(\xi_1,\eta_1,s)$ and $(\xi_2,\eta_2,s)$ can both be solutions of 
\eqref{zero:eig:torus} if and only if $\xi_1=\pm \xi_2,\eta_1=\pm\eta_2$ (otherwise 
$\pi$ would be a root of a nontrivial algebraic equation).    
It follows that for solutions of \eqref{zero:eig:torus} and 
\eqref{square:root} the values of $\xi^2$ and $\eta^2$ are fixed.  

Accordingly, the dimension of the kernel of $P_{1,g_s}$ is equal to either  
two (when $(\xi,\eta)=(0,\pm b)$ or $(\xi,\eta)=(\pm a,0)$) or four 
(when $(\xi,\eta)=(\pm a,\pm b), a\neq 0,b\neq 0$).  It is now easy to describe 
some of the corresponding nodal sets.  

In case $\xi=0,\eta=\pm a,s=4a,a\in\natls$ the eigenfunction is of the form 
$\sin 2\pi(ay+\theta)$.  Accordingly, up to translation in $y$, 
the nodal set is 
\begin{equation}\label{y:set}
[0,1]\times\{k/(2a):0\leq k\leq 2a\}\times [0,1].  
\end{equation}

In case $\eta=0,\xi=\pm a^2,s=2a,a\in\natls$, the eigenfunction is of the form 
$\sin 2\pi(a^2x+\theta)$.  Accordingly, up to translation in $x$, 
the nodal set is 
\begin{equation}\label{x:set}
\{k/(2a^2):0\leq k\leq 2a^2\}\times[0,1]\times [0,1].  
\end{equation}

Finally, in case $\xi=\pm a,\eta=\pm b, s^2=8(\eta^2+\sqrt{\eta^4+\xi^4/4}); 
a\neq 0,b\neq 0$, there exist ``product'' eigenfunctions of the form 
$\sin 2\pi(ax+\theta_1)\cdot \sin 2\pi(by+\theta_2)$ with nodal sets is 
a union of sets of the form \eqref{x:set} and \eqref{y:set}.  In addition, 
there exist eigenfunctions of the form $\sin 2\pi(ax+by+\theta)$, whose 
nodal sets (up to translation) have the form 
$$
\{(x,y)\in[0,1]^2:2(ax+by)\in\zed\}\times [0,1].  
$$


\subsection{The Paneitz operator}
Let us now look at the Paneitz operator~(\ref{eq:Paneitz-operator}). We have
  \begin{equation}
    P_{2,g_{s}}:=\Delta_{g_{s}}^{2}+\delta  V d+\frac{2d-3}{2}\left\{
    \frac{1}{4d}\Delta_{g_{s}}R_{g_{s}}+\frac{2d+1}{8(2d)^{2}}R_g^{2}-2|S|^{2}\right\},
    \label{eq:Paneitz-operator-gs}
\end{equation}where $S=\frac{1}{2d-1}(\op{Ric}_{g_{s}}-\frac{R_{g_{s}}}{2(2d)}g_{s})$
 is the Schouten-Weyl tensor and $V$ is the tensor $V=\frac{2d-1}{2(2d)} R_{g_{s}} g_{s}-4S$ acting on 1-forms. 

Using~(\ref{eq:Ricci-Heisenberg})--(\ref{eq:Scalar-Heisenberg}), we see that the Schouten-Weyl tensor is given by
  \begin{align*}
      S&=-\frac{3}{8(2d-1)}s^{2d+2}g_{s}+\frac{d+1}{2(2d-1)}s^{4d+2}\theta\otimes \theta \\ &=
      -\frac{3}{8(2d-1)}s^{2d+2}\sum_{1\leq j \leq d}\left(dx_{j}\otimes dx_{j}+s^{-2}dy_{j}\otimes dy_{j}\right) +
      \frac{4d+1}{8(2d-1)}s^{2d+2} \cdot s^{2d}\theta\otimes \theta .
  \end{align*}Observing that $\{dx_{j}\otimes dx_{j}, s^{-2}dy_{j}\otimes dy_{j},s^{2d}\theta\otimes \theta\}$ is an
  orthonormal family of $(0,2)$-tensors we find that
  \begin{equation*}
      |S|^{2}=\frac{16d^{2}+18d+1}{64(2d-1)^{2}}s^{4d+4}.
  \end{equation*}We then deduce that the constant coefficient of $P_{2}$ is equal to
  \begin{equation*}
    (2d-3)\frac{(2d+1)(2d-1)^{2}-4(16d^{2}+18d+1)}{256(2d-1)^{2}}s^{4d+4}.
  \end{equation*}

The tensor $V$ is given by
  \begin{equation*}
      V=\frac{12-(2d-1)^{2}}{8(2d-1)}s^{2d+2}g_{s}-\frac{2(d+1)}{2d-1}s^{4d+2}\theta\otimes \theta.
  \end{equation*}We need to look at $V$ as acting on 1-forms. The action of $g_{s}$ on 1-form is 
just the identity. The
  action of $s^{2d}\theta \otimes \theta$ is the orthogonal projection onto the span of $\theta$. Thus,
  \begin{equation*}
      \delta Vd=\frac{12-(2d-1)^{2}}{8(2d-1)}s^{2d+2}\Delta_{g_{s}}+2\frac{d+1}{2d-1}s^{2}T^{2}.
  \end{equation*}

Combining all this together we get

\begin{proposition}The Paneitz operator on $M$ for the metric $g_{s}$ is given by
\begin{equation*}
    P_{2,g_{s}}=\Delta_{g_{s}}^{2}-c_{1}(d)s^{2d+2}
\Delta_{g_{s}}+2\frac{d+1}{2d-1}s^{2}T^{2} +
   c_{0}(d)s^{4d+4},
\end{equation*}
where we have set
\begin{equation*}
  c_{0}(d):=(2d-3)\frac{(2d+1)(2d-1)^{2}-4(16d^{2}+18d+1)}{256(2d-1)^{2}} \quad \text{and} \quad  c_{1}(d):=\frac{(2d-1)^{2}-12}{8(2d-1)}.
\end{equation*}
\end{proposition}

Observe that $T\chi_{(\xi,\eta)}=0$ and 
$T^{2}W_{n}^{a,b}f_{\alpha}=-2n^{2}\pi^{2}$. Therefore, we can use the spectral
resolution of Laplacian $\Delta_{g_{s}}$ given by Remark~\ref{rmk:spectrum-Laplacian-Heisenberg} and to get a spectral 
resolution of $P_{2,g_{s}}$. 

\begin{proposition}\label{prop:spectrum-Paneitz-Heisenberg}
 Assume $M=\Gamma_{r}\backslash \bH_{d}$ is equipped with the metric $g_{s}$ given
    by~(\ref{eq:Heisenberg.metric-gs}).
    \begin{enumerate}
        \item  The family~(\ref{eq:Heisenberg.eigenfunctions}) for an orthogonal eigenbasis of $L^{2}(M)$ for $P_{2,g_{s}}$.

        \item  Each character $\chi_{(\xi,\eta)}$ is an eigenfunction of $P_{2,g_{s}}$ 
with eigenvalue 
        \begin{equation*}
           \lambda_{0}(\xi,\eta)^{2}-c_{1}(d)s^{2d+2}\lambda_{0}(\xi,\eta)+
   c_{0}(d)s^{4d+4},
        \end{equation*}where $\lambda_{0}(\xi,\eta)$ is given by~(\ref{eq:spectrum-Laplacian-Heisenberg-l0}). 

        \item  Each function $W_{n}^{a,b}f_{\alpha}$ is an eigenfunction of $P_{2,g_{s}}$ 
with eigenvalue
        \begin{equation*}
            \mu_{0}(n,|\alpha|)^{2}-c_{1}(d)s^{2d+2}
\mu_{0}(n,|\alpha|)+c_{0}(d)s^{4d+4}-4\frac{d+1}{2d-1}n^{2}\pi^{2}s^{2},
        \end{equation*}where $\mu_{0}(\xi,\eta)$ is given by~(\ref{eq:spectrum-Laplacian-Heisenberg-l0}). This eigenvalue has multiplicity $\displaystyle |n|^{d}|\Gamma_{r}|
        \binom{|\alpha|+d-1}{d-1}$ in $\cH_{n}$.
    \end{enumerate}
\end{proposition}

Set $F_{n}(\mu;s):= \mu^{2}-c_{1}(d)s^{2d+2}\mu+c_{0}(d)s^{4d+4}-4\frac{d+1}{2d-1}n^{2}\pi^{2}s^{2}$. Then 
Proposition~\ref{prop:spectrum-Paneitz-Heisenberg} states that
$W_{n}^{a,b}f_{\alpha}$ is an eigenfunction of $P_{2,g_{s}}$ with eigenvalue $F_{n}\left( \mu_{0}(n,|\alpha|);s\right)$. 
Moreover, $F_{n}(\mu;s)$ is a quadratic polynomial in $\mu$ with discriminant
\begin{equation*}
    \delta_{n}(d,s):= \delta_{0}(d)s^{4d+4}+ 16\frac{d+1}{2d-1}n^{2}\pi^{2}s^{2}, \qquad 
    \delta_{0}(d)=c_{1}(d)^{2}-4c_{0}(d).
\end{equation*}A computation shows that $\delta_{0}(d)=\frac{1}{4}\frac{4d^{2}-7}{2d-1}$, which is positive for $d \geq 
2$. 

Assume $d\geq 2$. Then $\delta_{n}(d,s)\geq \delta_{0}(d)s^{4d+4}>0$, and so $F_{n}\left( \mu_{0}(n,|\alpha|);s\right)$ 
is a negative eigenvalue of $P_{2,g_{s}}$ if and only if 
\begin{equation*}
    \mu_{0}(n,|\alpha|)<\frac{1}{2}\left( c_{1}(d)s^{2d+2}+ \sqrt{\delta_{n}(d,s)}\right).
\end{equation*}As $\delta_{n}(d,s)\geq \delta_{0}(d)s^{4d+4}$, we see that $\mu_{0}(n,|\alpha|)$ 
satisfies the above condition if $\mu_{0}(n,|\alpha|)<\frac{1}{2}\left( c_{1}(d) +\sqrt{\delta_{0}(d)s^{4d+4}}\right)$. That 
is, 
\begin{equation*}
    2 \pi |n|s (d+2|\alpha|)+4n^{2}s^{-2d}\pi^{2}<\frac{1}{2}\left( c_{1}(d) +\sqrt{\delta_{0}(d)}\right)s^{2d+2}.
\end{equation*}
This is the same type of condition than that incurring from~(\ref{eq:condition-null-eigenvalue-Paneitz-Heisenberg}) for $W_{n}f_{\alpha}^{a,b}$ to produce 
a negative eigenvalue of the Yamabe operator $P_{1,g_{s}}$. Therefore, by using the same kind of arguments as that
used to derive of Proposition~\ref{prop:neg-Yamabe-Heisenberg}, we obtain

\begin{proposition}\label{Paneitz:neg:heisenberg}
Assume $d\geq 2$. Then there is a constant $c_{d}>0$ depending on $d$, but not on the sequence 
$r=(r_{1},\cdots, r_{d})$, such that
\begin{equation*}
    \nu_{2}(g_{s})\geq c_{d}|\Gamma_{r}|s^{2d+2} \qquad \forall s >0. 
\end{equation*} 
In particular, for every integer $m\in \N$, the Paneitz operator $P_{2,g_{s}}$ has at least 
$m$ negative 
eigenvalues as soon as $s$ is large enough. 
\end{proposition}


\section{Open problems}\label{section: open problems}

\subsection{Discriminant hypersurfaces in the space
of conformal structures}\label{sec:discriminant}

Let $M$ be a compact orientable Riemannian manifold.  Denote by
$\cM$ the space of all Riemannian metrics on $M$.  Then one can
consider the action on $\cM$ of the group $\cP$ of (pointwise)
conformal transformations (multiplication by positive functions), as
well as of the group $\cD$ of diffeomorphisms; we shall denote by
$\cD_0$ the subgroup of $\cD$ of diffeomorphisms isotopic to
identity.

The objects considered in the paper are invariant under the action
of $\cP$, and equivariant with respect to the action of $\cD$.
Accordingly, it seems natural to consider our invariants as functions 
on the {\em Teichm\"uller space of conformal structures}
$$
\cT_M\ =\ \frac{\cM/\cP}{\cD_0},
$$
or on {\em Riemannian moduli space of conformal structures}
$$
\cR_M\ =\ \frac{\cM/\cP}{\cD},
$$
in the terminology of Fischer and Monkrief, \cite{FM96,FM97}. If $M$
is an orientable two-dimensional manifold, then $\cT_M$ (resp.
$\cR_M$) is the usual Teichm\"uller (resp. moduli) spaces. In
\cite{FM97}, the space $\cT_M$ for Haken $3$-manifolds $M$ of degree
$0$ is proposed as a configuration space for a Hamiltonian reduction
of Einstein's vacuum field equations.

\subsection{Dimension of the nullspace of a non-critical GJMS operator}
The Dirac operator is another important conformally invariant operator. 
Results of Maier~\cite{Ma:GMCSSM} in dimension 3 and Amman-Dahl-Humbert~\cite{ADH:SHS} in higher dimension show 
that, on a compact Riemannian spin manifold, and for a generic metric, the dimension of the nullspace 
of the Dirac operator  is equal to the lower bound provided by the Atiyah-Singer index theorem. In particular, 
the nullspace of the Dirac operator is generically trivial when $n\in \{3,4,5,6,7\} \bmod 8$. 

Let $k \in \N$ and further assume $k<\frac{n}{2}$ when $n$ is even. For the GJMS operator $P_{k}$ we make the following conjecture. 

\begin{conjecture}\label{conj:dimkerPk}
For a generic conformal class in $\cT_{M}$ the nullspace of $P_{k}$ is trivial.
\end{conjecture}

This conjecture will be addressed in the sequel~\cite{CGJP}. For a general  (possibly non-generic) conformal class, we mention the following conjecture due to Colin Guillarmou~\cite{Gu}.  

%
 
\begin{conjecture}[Guillarmou] 
Assume $n$ odd.  Then, for any conformal class in $\cT_{M}$, there exists $C>0$  such that
\[
\dim \ker P_{k} \leq C k^n \qquad \forall k \in \N.
\]
\end{conjecture}

\subsection{Hypersurfaces in $\cT_M$, $\cR_M$ and rigidity of the nodal set}
Bearing in mind Conjecture~\ref{conj:dimkerPk}, we consider the {\em discriminant hypersurface} $\cH_k$ (in $\cT$ or $\cR$) consisting of conformal classes with nontrivial nullspace
$\ker P_k\neq 0$.

\begin{conjecture}
For a generic conformal class in $\cH_k$, the nullspace of $P_k$ has dimension~$1$.
\end{conjecture}

Notice that when $\dim \ker P_{k}=1$ the nodal set and the nodal domains are well-defined. We also observe that 
$\cH_{k}$ contains all conformal classes of Ricci-flat metrics, 
since for a Ricci-flat metric $g$ Eq.~(\ref{eq:GJMS-Einstein}) shows that $P_{k,g}=\Delta_{g}^{k}$, and hence $\ker 
P_{k,g}$ is equal to the space of constant functions.

The following inverse (rigidity) problem seems natural:

\begin{problem}\label{rigidity1}
Let $g$ be a metric such that $\dim\ker P_{k,g}=1$.   Does
the nodal set $\cN(\phi),\phi\in\ker P_{k,g}$, determine the
corresponding conformal class $[g]\in\cH_k$ uniquely (up to
diffeomorphisms)?  In other words, do our invariants separate points
in $\cH_k$?
\end{problem}

The following weaker (deformation rigidity) version of the previous
problem also seems interesting:
\begin{problem}\label{rigidity2}
Let $g$ be a metric such that $\dim\ker P_{k,g}=1$. Does
the nodal set $\cN(\phi),\phi\in\ker P_{k,g}$, determine {\em locally} the
conformal class $[g]$?  In other words, can we deform a conformal class without 
changing $\cN(\phi)$?  
\end{problem}

We remark that it seems quite natural to consider Problems
\ref{rigidity1} and \ref{rigidity2} on the spaces $\cT_M$ and $\cR_M$,  
since the action of $\cP$ preserves the nodal sets, and their
definition is equivariant with respect to the action of $\cD$.  The
first natural step in this direction seems to be

\begin{problem}
Let $g$ be a metric such that  $\dim \ker P_{k,g}\geq 1$. Determine the tangent space
$T_{g}\cH\subset T_{g}\cM$.
\end{problem}

\subsection{Dimension of the nullspace  of the critical GJMS operator}
Assume $n$ even. For the critical GJMS operator $P_{\frac{n}{2}}$ on a $n$-dimensional
manifold with  $n$ even, the constant function will always be in $\ker P_{\frac{n}{2}}$. 

\begin{conjecture}
    For a generic conformal class in $\cT_M$, the nullspace of $\ker P_{\frac{n}{2}}$ consists of constant functions. 
\end{conjecture}

This conjecture is true if $M$ admits an Einstein metric of positive scalar curvature. 

We define the discriminant hypersurface $\cH_{\frac{n}{2}}$ as the set of 
of conformal classes for which the dimension of the nullspace of $P_{\frac{n}{2}}$ is at least $2$. 

\begin{conjecture}
For a generic conformal class in $\cH_{\frac{n}{2}}$, the nullspace of $P_{\frac{n}{2}}$ has dimension~$2$.
\end{conjecture}

It seems interesting to study the geometry and topology of the various $\cH_k$ and
their complements in the space of all conformal classes of Riemannian
metrics on $M$.


\subsection{Negative eigenvalues and topology of spaces of metrics}

Recall that it was shown in Proposition \ref{neg:eig:large} that
on any compact manifold of dimension $n\geq 3$, for any $m>0$ there
exist metrics $g$ for which the Yamabe operator $P_{1,g}$ has at least $m$
negative eigenvalues.  We have also constructed examples of Riemannian 
manifolds for which there are analogous results for some higher order 
GJMS operators (cf. Theorem~\ref{CriticalGJMS:hyperbolic}  and 
Theorem~\ref{Paneitz:neg:heisenberg}).  


\begin{problem}\label{infinite:finite}
Let $k>1$, and let $M$ be a compact manifold of dimension $n\geq 3$. 
Can we find for every $m\in\N$ a metric $g_m$ on $M$ such that 
$P_{k,g_m}$ has at least $m$ negative eigenvalues?
\end{problem}

We remark that if the number of negative eigenvalues of
$P_{k,g}$, $k>1$, is uniformly from bounded above for every metric $g$ on
$M$, then the smallest such bound would be a {\em topological} invariant
of $M$.

On Yamabe-negative manifolds (which do not admit metrics of
nonnegative scalar curvature), we know that in every conformal class
there exists at least {\em one} negative eigenvalue of $P_{1,g}$.  For
such manifolds, the following question formulated in \cite{BD} seems
natural:

\begin{problem}\label{neg:below}
Let $M$ be a Yamabe-negative compact manifold of dimension $n\geq
3$.  Does there exist an integer $m_{0}\geq 2$ such that in {\em every}
conformal class on $M$, the Yamabe operator $P_{1,g}$ has {\em at least} $m_{0}$ negative
eigenvalues?
\end{problem}

It is known from the work of Gromov-Lawson \cite{GL,Ros06} that, on
many manifolds of dimension $n\geq 5$ (and on some manifolds of
dimension $4$), the space of Yamabe-positive metrics (with positive
scalar curvature, or equivalently without negative eigenvalues of
$P_{1,g}$) can have infinitely many connected components.

On the other hand, Lohkamp (\cite{Lo92}; see also \cite{Kat}) showed
that the space of metrics with negative scalar curvature is
connected and has trivial homotopy groups.  Therefore, the following
seems natural:

\begin{problem}\label{topology:negative}
Let $M$ be a compact manifold of dimension $n\geq 3$. Given integers $k$ and $m$, describe
the topology of the space of all metrics $g$ for which the GJMS operator $P_{k,g}$ has at most $m$
negative eigenvalues (i.e., $\lambda_{m+1}(P_{k,g})\geq 0$). In particular, is that space
connected?
\end{problem}


\appendix \setcounter{section}{-1}
\renewcommand{\thesection}{A}

\section*{Appendix by A.\ Rod Gover and Andrea Malchiodi.\\ Non-critical Q curvature
  prescription.\\ Forbidden functions arising from non-trivial nullspace.  
 }\label{section:appendix}

\subsection{Background}
Some literature and background concerning curvature prescription was
mentioned in Section~\ref{section:scalar:sign}.  Concerning the
problem of prescribing $Q=Q_{\frac{n}{2},g}$ on even manifolds: in
\cite{Gov10,MalchSIGMA} it was shown that if the manifold and
conformal structure is such that the related critical GJMS operator
$P_{\frac{n}{2}}$ has non-trivial kernel (i.e., it contains
non-constant functions), then global considerations show that large
classes of functions cannot arise as the $Q$-curvature for some metric
in the given conformal class.  Due to the curious properties of
Branson's $Q$-curvature it turns out that the arguments required in
\cite{Gov10} are mainly of a linear or quadratic nature and benefit
from a conformal invariant identified in \cite{BGaim,BG}.

The conformal prescription problem for the other (``non-critical'')
$Q$-curvatures is rather different (having polynomial instead of
exponential non-linearities).  Nevertheless we work with the cases
$k\neq \frac{n}{2}$ here (so we exclusively consider the non-critical
$Q$-curvatures) and show that there are again global obstructions to
prescription of certain functions, arising from the presence of
non-trivial GJMS kernel.  In the following $k$ is an integer from the
usual range for the GJMS operators except that we shall suppose
henceforth that $2k\neq n$ (so $k$ is a positive integer with
$2k\notin \{n,n+2,\cdots \}$).

Recall the expression (\ref{Qk:curv}) defining the $Q$-curvatures and that 
in the case $k=1$ we have $Q_1=
R_{g}/2(n-1)$, where $R_{g}$ is the usual scalar curvature. 
In general the quantity $Q_k$ in \nn{Qk:curv} is called the order $2k$
 (non-critical) $Q$-curvature; for simplicity we shall refer to this
as simply a $Q$-curvature.  As in the body of the article, for  $Q_k$ we may write $Q_{k,g}$ to emphasise
the dependence on the metric $g$; we similarly treat related
quantities.

\subsection{The Problem}

The $Q$-prescription problem is described in (\ref{map}).  The partial
differential equation governing this follows from the conformal
transformation of the $P_k$ operator, as discussed in Section
\ref{section:Qcurv}. We summarise the facts from there in a form
convenient for our current purposes.

For the conformal transformation of $P_{k,g}$ we have
\begin{equation}\label{cov}
P_{k,{\widehat{g}}} e^{\frac{2k-n}{2}\om}u
= e^{-\frac{2k+n}{2}\om}P_{k,g} u,
\end{equation}
where $\widehat{g}=e^{2\om}g$, $\om,u \in C^\infty(M,\R)$.
So if we take, in particular, $u$ to be the positive function 
$u=e^{\frac{n-2k}{2}\om}$,
then $ e^{\frac{2k-n}{2}\om}u =1$, 
and so we conclude
\begin{equation}\label{preqn}
P_{k,{\widehat{g}}} 1=u^{\frac{n+2k}{2k-n}} P_{k,g} u .
\end{equation}

Putting \nn{Qk:curv} and \nn{preqn} together we obtain the non-linear
equation governing \nn{map}:
\begin{equation}\label{genYbe}
\left(\delta S_{k,g}d +\dfrac{n-2k}2Q_{k,g}\right)u=\dfrac{n-2k}2 Q_{k,\widehat{g}}u^{\frac{n+2k}{n-2k}},
\end{equation}
where $u$ is an arbitrary positive function. This generalises the
well-known scalar curvature prescription equation which is the $k=1$
special case.

\subsection{Forbidden functions}
Denote by $\mathcal{C}$ a conformal class of metrics on $M$.  
We are interested in what functions we can, or cannot, land on with $
Q_{k,\widehat{g}}$, where $\widehat{g}\in\mathcal{C}$. Let us fix $k$ and drop it 
from the notation. So
henceforth $P=P_k$ and $Q=Q_k$ for some fixed $k$ with $2k\in 2\mathbb{Z}\setminus
\{n,n+2,n+4,\ldots \}$. 

A first obstruction one can obtain rather easily, as in \cite{KW}, is
that if for some $g \in \mathcal{C}$ $Q_g$ has a given sign, then it is not
possible to prescribe a function with the opposite sign: this follows
immediately by integrating \eqref{genYbe}. We notice first that for
the case $k=1$ the sign of the Yamabe invariant
$$
  \inf_{\widehat{g} \in \mathcal{C}} \frac{\int R_{\, \widehat{g}} \, dv_{\widehat{g}}}{\left( \int 
  \, dv_{\widehat{g}} \right)^{\frac{n-2}{n}}}
$$ 
coincides with the sign of the first eigenvalue of the conformal
  Laplacian and determines uniquely the possible sign of the scalar
  curvature for the metrics in $\mathcal{C}$. This is not the case in general
  for larger $k$, due to a lack of maximum principle.

We have next the following observation, more peculiar to the presence
of a kernel.  Here and subsequently we write $ \ker P_{g}$
for the kernel (or nullspace) of $P_g$.

\begin{proposition}\label{basic} Consider a closed manifold $M$ equipped with a conformal structure $\mathcal{C}$, and 
let $g \in \mathcal{C}$.  If $0\neq u\in \ker P_{g}$, then $u$ is not
in the range of $Q$. That is, $u\neq Q_{\widehat{g}}$ for all
$\widehat{g}\in \mathcal{C}$.
\end{proposition}

\begin{proof} Suppose with a view to contradiction that
$\widehat{g}\in \mathcal{C}$ and $Q_{\widehat{g}}=u$. Since $g,\widehat{g}\in \mathcal{C}$
we have $\widehat{g}=e^{2\om}g$ for some $\om\in C^\infty (M,\R)$.

Now from \nn{cov}, if $ u\in \ker P_{g}$ then $e^{f} u \in \ker
P_{\widehat{g}}$, where $f= {\frac{2k-n}{2}\om}$. So, using that
$P_{\widehat{g}}$ is formally self-adjoint, it follows that for any
function $v$ 
$$\int u e^f P_{\widehat{g}} v \, dv_{\widehat{g}} =0.
$$
Thus, taking $v=1$,  this shows
$$0= \frac{2}{n-2k}\int u e^f P_{\widehat{g}} 1 \, dv_{\widehat{g}} =\int u e^f
Q_{\widehat{g}} \, dv_{\widehat{g}}  = \int u e^f u \, dv_{\widehat{g}}  =\int u^2 e^f \, dv_{\widehat{g}}  .$$ This is a contradiction
since $ u^2 e^f$ is a non-zero non-negative function.
\end{proof}

More generally essentially the same argument shows that we cannot have
$Q_{\widehat{g}}$ equal $s_u$, where the latter is any function that
has the same or opposite strict sign as $u$: Suppose with a view to
contradiction $Q_{\widehat{g}}=s_u$. Then
$$
0= \frac{2}{n-2k}\int u
e^f P_{\widehat{g}} 1  \, dv_{\widehat{g}} 
 =\int u e^f s_u  \, dv_{\widehat{g}} 
= \int  e^f u s_u \, dv_{\widehat{g}}   ,
$$ 
which is impossible. Thus we have the following result. 

\begin{theorem}\label{main}
Consider a closed manifold $M$ equipped with a conformal structure $\mathcal{C}$, and 
let $g \in \mathcal{C}$.  Suppose there exists $u\in \ker P_{g}\setminus \{ 0\}$. 
Then for any function $s_u$ on $M$ with the same  or opposite strict sign as $u$, $s_u$ is not in the
range of $Q$. That is $s_u\neq Q_{\widehat{g}}$, for all $\widehat{g}\in
\mathcal{C}$.
\end{theorem}

Observe that if there exists $u$, as in the Theorem, then there is a
huge class of functions satisfying the conditions on $s_u$: for
example $e^f u^p$ where $p$ is an odd positive integer and $f\in C^\infty(M,\R)$.
We record this for emphasis.
\begin{corollary}\label{big}
If $P_g$ has non-trivial kernel then there is an infinite dimensional
space of functions disjoint from $\mathcal{R}(Q)$.
\end{corollary}

\subsection{Constraints on $\mathcal{R}(Q)$}\label{constraints}

To prove Theorem \ref{main} we used that given $u\in \ker P_g$ then for any 
$\widehat{g}\in \mathcal{C}$ we have 
\begin{equation}\label{int}
\int u e^f Q_{\widehat{g}} \, dv_g =0 
\end{equation} 
for some real function $f$, which depends on $\widehat{g}$. In fact if
$\widehat{g}=e^{2\om} g$ then $f= \frac{n+2k}{2}\om$, but the details
are not important. The key point here is that $e^f$ is a strictly
positive function, thus for $u$ non-zero the display captures some
non-trivial constraint on $\mathcal{R}(Q)$ as demonstrated in Theorem
\ref{main} and its Corollary.

Given elements $u\in \ker P_g$ and $g\in \mathcal{C}$, consider the linear form
$I_u^g: C^\infty (M,\R)\to \mathbb{R}$ defined by 
$$
I_{u}^{g}(v)= \int u v \, dv_g  \qquad \forall v \in C^\infty (M,\R).
$$

Now let us fix some $g\in \mathcal{C}$, and 
for the moment also fix some $u\in \ker P_g$.  From
\nn{int}, and the conformal transformation of the standard metric
measure, we have that if $v=Q_{\widehat{g}}$, for some $\widehat{g}\in
\mathcal{C}$, then there exists $g'\in \mathcal{C}$ such that
\begin{equation}\label{orth}
\int uv \, dv_{g'}=0.
\end{equation}
Thus 
$$
\mathcal{R}(Q)\subseteq \bigcup_{g' \in \mathcal{C}} \mathcal{Z}(I^{g'}_u) ,
$$ where $\mathcal{Z}(I^{g'}_u)$ denotes the kernel of the map
$I^{g'}_u$. This holds for all elements of $\ker P_g$, thus we have
\begin{equation}\label{rhs}
\mathcal{R}(Q)\subseteq \bigcap_{u'\in \ker P_g}
\left[ \bigcup_{g' \in \mathcal{C}} \mathcal{Z}(I^{g'}_{u'})\right] .
\end{equation}
By definition $\mathcal{R}(Q)$ depends only on the conformal
structure. On the other hand we had fixed $g\in \mathcal{C}$ to describe the
right-hand-side here. A different choice would result in each of the
elements $u'\in \ker P_g $ being replaced by a positive
function multiple $e^f u'$, with the same function $e^f$ for all
elements of $\ker P_g$. Examining \nn{orth}, we see that this
factor $e^f$ may be absorbed by moving to a conformally related
measure. Since we average over all such in the right-hand-side of
\nn{rhs} it is clear that in fact this function space is independent of
$g$, and depends only on $\mathcal{C}$.

Since $\ker P_g$ is finite dimensional and $I^g_{u'}$ is
linear in its dependence on $u'\in \ker P_g$ we obtain the following
refinement of the above.

\begin{theorem}\label{consthm}
On a closed conformal manifold $(M,\mathcal{C})$ let $g\in \mathcal{C}$. Then 
 $$ 
\mathcal{R}(Q)\subseteq \mathcal{I}= \bigcap_{i=1}^{\ell}
\left[ \bigcup_{g' \in c} \mathcal{Z}(I^{g'}_{u_i})\right] .
$$ where $\ell=\dim \ker P_g$, and $u_1,\cdots ,u_\ell$ is a
basis for $\ker P_g$. Furthermore the function space $\mathcal{I}$ is independent of the choice of $g\in \mathcal{C}$ and
the choice of basis $\{u_1,\cdots,u_\ell\}$.
\end{theorem}

An important special case is prescribing constant $Q$-curvature. This is
related to the Yamabe problem which seeks to find within a conformal
class a metric with constant scalar curvature.  Note that if there
is a metric $g\in \mathcal{C}$ such that $Q_g={\rm constant}\neq 0$ then it is
clear from \nn{orth} that any non-zero element of $\ker P_g$
must change sign on $M$. But this sign change property is independent
of $g\in \mathcal{C}$. Thus by contrapositive we have the following.

\begin{theorem}\label{constants}
On a connected conformal manifold $(M,\mathcal{C})$ let $g\in \mathcal{C}$. Suppose that
there is a non-zero function $u\in \ker P_g$ that is 
either everywhere non-negative, or
everywhere non-positive. Then there is no non-zero constant function in
$\mathcal{R}(Q)$.
 \end{theorem}
\begin{remark}
Note the result in the Proposition is interesting only if $u$, as
described there, has a non-trivial zero locus. Otherwise if 
$u$ is strictly positive or strictly
negative then it follows easily from \eqref{cov} that there is a 
metric $\widehat{g}\in
\mathcal{C}$ such that $P_{\widehat{g}}$ annihilates (all) constant
functions and so the constant function $0$ is in $\mathcal{R}(Q)$, see Proposition \ref{Pk:kernel:const}.
\end{remark}

\subsection{Final comments}

Although we have focussed on the GJMS operators the results here apply
more widely. We could replace the GJMS operators with any conformally
covariant and formally self-adjoint 
operator $\overline{P}_k$ (on functions) of the
same conformal bidegree and taking the form 
$$
\overline{P}_k= \delta \overline{S}_kd+ \overline{Q}_k,
$$ (cf., \nn{Qk:curv}) with $\overline{Q}_k$ non-trivial. Then the
$\overline{Q}_k$ prescription theory would mirror that above.
Alternative conformal powers of the Laplacian, with these properties,
are described in \cite{Gsrni} (due to M.G. Eastwood and the first
author). 

Similarly in certain circumstances the restriction $2k\notin \{
n+2,n+4,\cdots \}$, on even manifolds, can be relaxed. For example
this is the case if the manifold is locally conformally flat or if it
is locally conformally Einstein \cite{Gov06}; in both settings there is a
class of differential operators which extends the GJMS family to these
orders.

\section*{Acknowledgements}
This paper has its origin from discussions between J.D.\ and
A.R.G.\ during the 2010 Banff workshop on geometric scattering
theory. They wish to thank BIRS and the workshop organizers for their
hospitality. R.P.\ wishes to thank McGill University for its
hospitality during his various visits to Montreal. In addition, the
authors wish to thank Alice Chang, Penfei Guan, Colin Guillarmou, Niky
Kamram, Rafe Mazzeo, Peter Sarnak, Richard Schoen, and Paul Yang for
useful discussions related to he subject matter of this paper.



\begin{thebibliography}{CGJP1}
\bibitem[Al1]{Al03} S.\ Alexakis. \emph{On conformally invariant differential 
operators in odd dimension}. Proc. Nat. Acad. Sci. USA {\bf 100} (2003), no. 2, 4409--4410.

\bibitem[Al2]{Al06} S.\ Alexakis. \emph{On conformally invariant differential 
operators}. E-print, arXiv, Aug.~06, 50 pages. 

\bibitem[ADH]{ADH:SHS} B.\ Ammann, M.\ Dahl, and E.\ Humbert. \emph{Surgery and harmonic spinors}. Adv.\ in Math.\
\textbf{220} (2009), no. 2, 523--539.

\bibitem[Au]{Au} T.\ Aubin. {\em Some nonlinear problems in Riemannian geometry}, 
Springer Monographs in Mathematics, Springer-Verlag, Berlin, 1998.

 \bibitem[BD]{BD} C. B\"ar and M. Dahl. {\em Small eigenvalues of the
conformal Laplacian.} GAFA  {\bf 13} (2003) 483--508.

\bibitem[BFR]{BaFaRe} P.\ Baird, A.\ Fardoun and R.\ Regbaoui. {\em
    Prescribed $Q$-curvature on manifolds of even dimension}, J.\ Geom.\
  Phys.\ {\bf 59} (2009), 221--233.
 
\bibitem[BS]{BS} E.\ Bogomolny and C.\ Schmit. \emph{Percolation model for nodal domains of chaotic wave functions}, 
Phys.\ Rev.\ Lett.\ \textbf{88} (2002), 114102.

\bibitem[BE]{B} J.-P.\ Bourguignon and J.-P. Ezin. {\em Scalar curvature
  functions in a conformal class of metrics and conformal
  transformations}, Trans.\ Amer.\ Math.\ Soc.\ {\bf 301} (1987), 
  723--736.

\bibitem[Br1]{Br} T. Branson. {\em Differential operators canonically associated to a
conformal structure.} Math.\ Scand.\ {\bf  57} (1985), 293--345.

\bibitem[Br2]{BrSeoul} T.P.\ Branson. {\em The functional
  determinant}. Lecture Notes Series, 4. Seoul National University,
  Research Institute of Mathematics, Global Analysis Research Center,
  Seoul, 1993, vi+103 pp.

\bibitem[Br3]{Tomsharp} T.P.\ Branson. {\em  Sharp inequalities, the
  functional determinant, and the complementary series},
  Trans.\ Amer.\ Math.\ Soc.\  {\bf  347} (1995), 3671--3742.
  
\bibitem[BCY]{BCY} T.P.\ Branson,  S.-Y. A. Chang, and P. Yang. \emph{Estimates and extremals for zeta function
determinants on four-manifolds}, Commun. Math. Phys. \textbf{149} (1992), 241--262.
  
\bibitem[BG1]{BGaim} T.P.\ Branson and A.R.\  Gover. {\em Origins, applications and generalisations of the $Q$-curvature}, American
  Institute of Mathematics, (2003), \hfill
  http://www.aimath.org/pastworkshops/confstruct.html

\bibitem[BG2]{BG} T. Branson and A. Rod Gover. {\em Origins, applications
and generalisations of the $Q$-curvature.} Acta Appl. Math. {\bf 102} 
(2008), no. 2-3, 131--146.

\bibitem[B{\O}1]{BrO91a} T.P.\ Branson and B.\ \O rsted. \emph{Conformal geometry and global invariants}, Differential
Geom. and Appl. \textbf{1} (1991), 279--308.

\bibitem[B{\O}2]{BrO} T.P.\ Branson and B.\ \O rsted.   
{\em Explicit functional determinants in four dimensions},
Proc.\ Amer.\ Math.\ Soc., {\bf 113} (1991), 669--682. 

\bibitem[Bren]{Brendle} S.\ Brendle. {\em Convergence of the $Q$-curvature flow
  on $S\sp 4$},  Adv.\ Math.\  {\bf 205} (2006), 1--32.

\bibitem[Brez]{Bre} J. Brezin. \emph{Harmonic analysis on nilmanifolds}. Trans. Am. Math. 
Soc. {\bf 150} (1970) 611--618.

\bibitem[Bu]{Buser} P. Buser. {\em Riemannsche Fl\"achen mit Eigenwerten in
$(0,1/4)$.} Comment. Math. Helvetici {\bf 52} (1977), 25--34.


\bibitem[CGJP1]{ERA} Y. Canzani, R. Gover, D. Jakobson and R. Ponge. 
{\em Nullspaces of Conformally Invariant Operators. Applications to $Q_{k}$-curvature.} E-print, arXiv, June 2012. 

\bibitem[CGJP2]{CGJP} Y. Canzani, R. Gover, D. Jakobson and R. Ponge. 
{\em Conformal invariants from nodal sets. II.} Preprint 2012.  

\bibitem[CP]{CP} Y. Canzani and R. Ponge. \emph{An Ulhenbeck theorem for conformally invariant pseudo-differential 
operators}. Preprint 2012. 

\bibitem[CGY]{CGY} S.-Y.A.\ Chang,  M.J.\ Gursky and P.C.\  Yang. {\em A 
conformally invariant sphere theorem in four dimensions},  Publ.\ Math.\ Inst.\
 Hautes \'Etudes Sci.\  {\bf 98}  (2003), 105--143. 

\bibitem[CY]{CY} S.-Y.A.\ Chang and P.C.\ Yang. {\em Extremal metrics of
  zeta function determinants on 4-manifolds}, Ann.\ of Math.\ (2) {\bf 142}
  (1995), no.\ 1, 171--212.

\bibitem[Ch]{Ch:ENS} S.Y.\ Cheng. \emph{Eigenfunctions and nodal sets}. Comment.\ Math.\ Helv.\ \textbf{51} (1976), no. 1, 43--55.

\bibitem[CYau]{Cheng-Yau} S.-Y.\ Cheng and S.-T.\ Yau. \emph{Differential equations on Riemannian manifolds and their geometric applications}.
Comm.\ Pure Appl.\ Math.\ \textbf{28} (1975), no. 3, 333--354. 

\bibitem[Ch]{Ch3} V. Chernov. {\em Framed knots in 3-manifolds and 
affine self-linking numbers.} J. Knot Theory Ramifications 14 (2005), {\bf 6}, 791--818.

\bibitem[CR]{CR} V. Chernov and Y. Rudyak. {\em Toward a general theory of linking 
invariants.} Geom. Topol. {\bf 9} (2005), 1881--1913.  

\bibitem[CH]{Courant} R.\ Courant and D.\ Hilbert. \emph{Methods of mathematical physics. Vol. I.} Interscience Publishers, Inc., New York, N.Y.,1953. xv+561 pp..

\bibitem[DM]{DM} Z.\ Djadli and A.\ Malchiodi. 
{\em Existence of conformal metrics with constant $Q$-curvature}, 
Ann.\ of Math. (2), {\bf  186} (2008), 813--858.

\bibitem[DR]{DR}
P.\ Delano\"e and F.\ Robert.  
{\em On the local Nirenberg problem for the $Q$-curvatures},
Pacific J.\ Math., {\bf  231} (2007), 293--304. 

\bibitem[DF]{DF} H.\ Donnelly and C.\ Fefferman. \emph{Nodal sets of eigenfunctions on Riemannian manifolds}. Invent. 
Math. \textbf{93} (1988), no. 1, 161--183.

\bibitem[Di]{Dirac} P.A.M.\ Dirac.  \emph{Wave equations in conformal space}. Ann.\ of Math. \textbf{37} (1936), 429--442.

\bibitem[ES]{ES} M.G.\ Eastwood and J.\ Slov\'ak. {\em Semiholonomic Verma modules}. J. Algebra 197 (1997), no. 2, 424--448.

\bibitem[FR]{FR} A.\ Fardoun and R.\ Regbaoui. {\em $Q$-curvature flows for GJMS 
operators with nontrivial kernel.}  E-print, arxiv:1203.3035

\bibitem[Fe]{Fe:GMT} H. Federer. \emph{Geometric measure theory}. Die Grundlehren der mathematischenWissenschaften, 
Vol.\ 153,  Springer-Verlag, 1969, 680 pages. 

\bibitem[FG1]{FG85} C. Fefferman and C.R. Graham. \emph{Conformal invariants}.
\'Elie Cartan et les Math\'ematiques d'Aujourd'hui, Ast\'erisque,
hors s\'erie, (1985), 95--116.


\bibitem[FG2]{FG02} C. Fefferman and C.R. Graham.  {\em $Q$-curvature and Poincar\'e
metrics.}  Math. Res. Lett. {\bf 9} (2002), 139--151.

\bibitem[FG3]{FG12} C. Fefferman and C.R. Graham, The ambient metric,
  {\em Annals of Mathematics Studies}, 178, Princeton University
  Press, Princeton, NJ, 2012. x+113 pp.  

\bibitem[FM1]{FM96} A. Fischer and V. Moncrief. {\em The structure of quantum
conformal superspace,} in Global Structure and Evolution in General
Relativity, ed. by S. Gotsakis and G. Gibbons, Springer (1996),
111-173.

\bibitem[FM2]{FM97} A. Fischer and V. Moncrief. {\em Hamiltonian
reduction of Einstein's equations of general relativity.} Nucl.
Physics B (Proc. Suppl.) {\bf  57} (1997), 142--161.

\bibitem[Fo]{Fo} G. Folland. \emph{Compact Heisenberg manifolds as CR
manifolds}. J. Geom. Anal. { \bf 14} (2004) 521--532.

\bibitem[GH]{GH} A.R. Gover and K.~Hirachi. \emph{Conformally invariant powers of the Laplacian -- 
A complete non-existence theorem}. J. Amer. Math. Soc. { \bf 17} (2004), 389--405.  

\bibitem[Go1]{Gsrni} A.R.\ Gover.  {\em Aspects of parabolic invariant
theory.}   {\em The 18th Winter
School ``Geometry and Physics'' (Srn\'{i} 1998)}, pp. 25--47.
Rend.\ Circ.\ Mat.\ Palermo (2) Suppl.\ {\bf 59}, 1999.

\bibitem[Go2]{Gov06} A. R. Gover. {\em Laplacian operators and 
$Q$-curvature on
conformally Einstein manifolds.} Math. Ann. {\bf 336} (2006), no. 2,
311--334.

\bibitem[Go3]{Gov10} A. R. Gover. {\em Q curvature prescription; forbidden
functions and the GJMS null space.} Proc. Amer. Math. Soc. {\bf 138} 
(2010), No. 4, 1453--1459.


\bibitem[GP]{GP} A.R.~Gover and L.J.~Peterson. \emph{Conformally Invariant Powers of the Laplacian,
$Q$-Curvature, and Tractor Calculus}. Comm. Math. Phys. {\bf 235} (2003), 339--378. 

\bibitem[GW]{GW} C.S. Gordon and E.N. Wilson.  
{\em The spectrum of the Laplacian on Riemannian Heisenberg manifolds.}
The Michigan mathematical journal {\bf 33} (1986), no. 2,  253--271.

\bibitem[Gr1]{Gnon} C.R.\ Graham.  {\em Conformally invariant powers
  of the Laplacian. II. Nonexistence}. J. London Math. Soc. (2) {\bf 46} 
  (1992), no. 3, 566--576.

\bibitem[Gr2]{Gr03} C.R.\ Graham.  Talk at the workshop
  \emph{Conformal structure in geometry, analysis, and physics}.  AIM,
  Palo Alto, Aug.\ 12--16, 2003.
  
  \bibitem[GJMS]{GJMS} C.R. Graham, R. Jenne, L.J. Mason and G.A. Sparling.
{\em Conformally invariant powers of the Laplacian, I: Existence.}
J. Lond. Math. Soc. {\bf 46} (1992), 557--565.

\bibitem[GL]{GL} M. Gromov and H. Blaine Lawson.  {\em The classification
of simply connected manifolds of positive scalar curvature.} Ann. of
Math. (2) {\bf 111} (1980), no. 3, 423--434.

\bibitem[GZ]{GZ} C.R.\ Graham and M. Zworski. {\em Scattering matrix in conformal
geometry.}  Invent. Math. {\bf 152} (2003), 89--118.

\bibitem[Gu]{Gu} C.\ Guillarmou. Private communication.

\bibitem[GN]{GN} C.\ Guillarmou and F.\ Naud. \emph{Wave 0-trace and length spectrum on convex co-compact hyperbolic manifolds}. 
Comm. Anal. Geom. \textbf{14} (2006), no. 5, 945--967.

\bibitem[GQ]{GQ} C.\ Guillarmou and  J.\ Qing. 
\emph{Spectral characterization of Poincar\'e-Einstein manifolds with infinity of positive Yamabe type}, 
Int. Math. Res. Not. \textbf{2010}, no. 9, 1720--1740. 

 \bibitem[Ha]{H:SESEE} Q.\ Han. \emph{Singular sets of solutions to elliptic equations}. Indiana Univ. Math. J. \textbf{43} 
 (1994), no. 3, 983--1002
 
 \bibitem[HHL]{HHL:GMSSEE} Q.\ Han, R.\ Hardt, and F.H.\ Lin. \emph{Geometric measure of singular sets of elliptic 
 equations}. Comm.\ Pure Appl.\ Math.\ \textbf{51} (1998), no. 11-12, 1425--1443.
 
 \bibitem[HHHN]{HHHN:CSSEE} R.\ Hardt, M.\ Hoffmann-Ostenhof, T.\ Hoffmann-Ostenhof, and N.\ Nadirashvili. 
 \emph{Critical sets of solutions to elliptic equations}. J. Differential Geom.\ \textbf{51} (1999), no. 2, 359--373.

\bibitem[Je]{Jen} G. Jensen. {\em
The scalar curvature of left invariant Riemannian metrics.} 
Indiana Univ. Math. J. {\bf 20} (1971), no 12, 1125--1144.

\bibitem[Ju]{Ju} A. Juhl. \emph{Families of conformally covariant 
differential operators, $Q$-curvature and holography}. Progress in
Mathematics, Birkh\"auser, Vol. 275, 2009, 500 pages.

\bibitem[Ka]{Kat} M. Katagiri. {\em On the topology of the moduli
space of negative constant scalar curvature metrics on a Haken
manifold.} Proc. Japan Acad. {\bf 75} (A), 126--128.

\bibitem[KW1]{KW} J.L.\ Kazdan and F.W.\ Warner. {\em Curvature functions
  for compact 2-manifolds}, Ann.\ of Math.\ (2) {\bf 99} (1974), 14--47.

\bibitem[KW2]{KW75} J. Kazdan and F. Warner. {\em Scalar curvature
and conformal deformations of Riemannian structure.} J. Diff. Geom.
{\bf 10} (1975), 113--134.

\bibitem[KS]{KS:DCAS3} K. Kodaira, and D. Spencer. \emph{On deformations of complex analytic structures, III. Stability
theorems for complex structures}. Ann.\ of Math., \textbf{71} (1960), no.\ 1, 43--76.


\bibitem[Lo1]{Lo92} J. Lohkamp. {\em The space of negative scalar curvature
metrics.}  Invent. Math. {\bf 110} (1992), 403--407.

\bibitem[Lo2]{Lo96} J. Lohkamp. {\em Discontinuity of geometric expansions.}
Comment. Math. Helvetici {\bf 71} (1996) 213--228.

\bibitem[Ma]{Ma:GMCSSM} S.\ Maier. \emph{Generic metrics and connections on Spin- and $\text{Spin}^{c}$-manifolds}, Comm.\ 
Math.\ Phys.\ \textbf{188} (1997), no. 2, 407--437.

\bibitem[Mal]{MalchSIGMA} A.\ Malchiodi. {\em Conformal metrics with constant
  $Q$-curvature},  SIGMA Symmetry Integrability Geom. Methods Appl.  {\bf 3}
  (2007), Paper 120, 11 pages. 

\bibitem[MS]{MalStr} A.\ Malchiodi and M.\ Struwe. {\em $Q$-curvature
  flow on $S\sp 4$}, J.\ Differential Geom.\ {\bf 73} (2006), 1--44.

\bibitem[NS]{NS} F.\ Nazarov and M.\ Sodin. \emph{On the number of nodal domains of random spherical harmonics}. 
Amer. J. Math. \textbf{131} (2009), no. 5, 1337--1357.

\bibitem[Nd]{Nd:CQMAD} C.B.\ Ndiaye. \emph{Constant $Q$-curvature metrics in arbitrary dimension.} J.\ Funct.\ Anal.\ \textbf{251} (2007), no. 
1, 1--58.

\bibitem[Ok]{Ok} K.\ Okikiolu. \emph{Critical metrics for the determinant of the Laplacian in odd dimensions}, Ann. of 
Math. (2) \textbf{153} (2001), no. 2, 471--531.
  
  \bibitem[OPS]{OPS} B.\ Osgood, R.\ Phillips, and
    P.\ Sarnak. \emph{Extremals of determinants of Laplacians},
    J. Funct. Anal. \textbf{80} (1988), 148--211.

\bibitem[Pa]{Paneitz} S.\ Paneitz. {\em A quartic conformally covariant
differential operator for arbitrary pseudo-Riemannian manifolds}.
Preprint (1983). Reproduced as: SIGMA {\bf 4} (2008), 036, 3 pages 

\bibitem[PR]{PR} T. Parker and S. Rosenberg.  \emph{Invariants of conformal Laplacians}. 
  J. Differential Geom. \textbf{25} (1987), no. 2, 199--222.

\bibitem[Pl]{Pleijel} A.\ Pleijel. \emph{Remarks on Courant's nodal line theorem}. Comm. Pure Appl. Math. \textbf{9} (1956), 543--550.

\bibitem[Ro]{Ros06} J. Rosenberg. {\em Manifolds of positive
scalar curvature: a progress report.} Surveys in differential
geometry. Vol. XI, 259--294, Surv. Differ. Geom., 11, Int. Press,
Somerville, MA, 2007.

\bibitem[TZ]{TZ} J.\ Toth and S.\ Zelditch. \emph{Counting nodal lines which touch the boundary of an analytic domain}. 
J. Differential Geom. \textbf{81} (2009), no. 3, 649--686.

\bibitem[W\"u]{Wu} V.\ W\"{u}nsch. \emph{On conformally invariant differential operators}. Math. Nachr. \textbf{129} (1986), 269--281.

\bibitem[SZ]{SZ:LBHMNS} C.\ Sogge and S.\ Zelditch. \emph{Lower bounds on the Hausdorff measure of nodal sets}. 
Math.\ Res.\ Lett.\ \textbf{18} (2011), no. 1, 25--37.
\end{thebibliography}
\end{document}